\documentclass{article}
\usepackage{graphicx}
\usepackage[utf8]{inputenc}
\usepackage{amsfonts}
\usepackage{amsthm}
\usepackage{amsmath}
\usepackage{amssymb}
\usepackage{tikz}
\usepackage{tikz-cd}
\usepackage{parskip}
\usepackage{tikz}
\usetikzlibrary{patterns,shapes.symbols}

\title{Superheavy Skeleta for non-Normal Crossings Divisors}
\date{}
\author{Elliot Gathercole}
\numberwithin{equation}{section}
\numberwithin{figure}{section}
\begin{document}

\maketitle

\theoremstyle{theorem}
\newtheorem{prop}{Proposition}[section]
\newtheorem{thm}[prop]{Theorem}
\newtheorem{cor}[prop]{Corollary}
\newtheorem{lem}[prop]{Lemma}

\theoremstyle{definition}
\newtheorem{defn}[prop]{Definition}
\newtheorem{set}[prop]{Setting}
\newtheorem{ex}[prop]{Example}

\tableofcontents

\bibliographystyle{plain}

\section{Introduction}

\subsection{Motivation}
The purpose of this paper is to prove non-displaceability results for certain, possibly singular or non-orientable, Lagrangians in monotone Kähler surfaces. In particular, will prove that these Lagrangians are superheavy with respect to the fundamental class (all results will be independent of ground field), which implies non-displaceability.

Consider the monotone Lagrangian Klein bottle $K\subset \mathbb{CP}^1\times \mathbb{CP}^1$ described in \cite{Ev}. We would like to know if $K$ is displaceable from itself. If the Floer cohomology of a Lagrangian is well-defined and non-zero, it gives an obstruction to displaceability. However, $K$ bounds discs of Maslov index 1, which complicate matters, and a naive calculation appears to indicate that $K$ is not even weakly unobstructed, so it appears that $HF^*(K)$ is not well-defined. Even so, we will prove the following, stronger result about $K$.
\begin{thm}
\label{Klein_bottle}
Consider $\mathbb{CP}^1\times\mathbb{CP}^1$ as a product symplectic manifold, where each factor is equipped with the Fubini-Study symplectic form, and has area $1$. 

For $i=1,2$, let $[x_i:y_i]$ be coordinates on the $i$-th factor of $\mathbb{CP}^1\times\mathbb{CP}^1$.

Let
\begin{equation}
p_i=\frac{-|x_i|^2+|y_i|^2}{2(|x_i|^2+|y_i|^2)}
\end{equation}
and
\begin{equation}
\theta_i=\arg(\overline{x}_iy_i)
\end{equation}
so $(p_i,\theta_i)$ are action-angle coordinates on the $i$-th factor.

The embedded Lagrangian Klein bottle
\begin{equation}
K=\{p_1=2p_2, 2\theta_1+\theta_2=0\}\subset\mathbb{CP}^1\times\mathbb{CP}^1
\end{equation}
is superheavy.
\end{thm}

We could ask the same question about the displaceability of singular Lagrangians, like the subset $\{[z_0:z_1:z_2]\in\mathbb{CP}^2| \overline{z_0}z_1=\pm z_0\overline{z_2} \}$, a union of two Lagrangian $\mathbb{RP}^2$s which intersect transversely at a point, and cleanly along a circle. Defining the self Floer cohomology of such a union is much more involved than for a single submanifold. We will show the following, more general result (the above example is the case $\alpha=\pi$).
\begin{thm}
\label{RP2s}
Consider $\mathbb{CP}^2$ as a symplectic manifold, equipped with the Fubini-Study symplectic form, normalised so a line has area $1$. 

Let $[z_0:z_1:z_2]$ be the standard coordinates on $\mathbb{CP}^2$. 

Let
\begin{equation}
p_i=\frac{|z_i|^2}{|z_0|^2+|z_1|^2+|z_2|^2}
\end{equation}
and
\begin{equation}
\theta_i=\arg(\overline{z}_0z_i),
\end{equation}
so $(p_i,\theta_i)$ for $i=1,2$ are action-angle coordinates.

For $\alpha\in\mathbb{R}/2\pi\mathbb{Z}$,
\begin{equation}
L_\alpha=\{p_1=p_2, \theta_1+\theta_2=\alpha\}\subset\mathbb{CP}^1\times\mathbb{CP}^1
\end{equation}
is an embedded Lagrangian $\mathbb{RP}^2$. For $\alpha\in \left[\frac{2\pi}{3},\frac{4\pi}{3}\right]$, $L_0\cup L_\alpha\subset\mathbb{CP}^2$ is superheavy. 	

Note that for values of $\alpha$ outside this range, the union $L_0\cup L_\alpha$ is disjoinable from a Chekanov torus, and therefore cannot be superheavy.
\end{thm}

We could also consider an analogous example in three dimensions.
\begin{thm}
\label{RP3s}
Consider $\mathbb{CP}^3$ as a symplectic manifold, equipped with the Fubini-Study symplectic form, normalised so a line has area $1$. 

Let $[z_0:z_1:z_2:z_3]$ be homogeneous coordinates on $\mathbb{CP}^3$.

Let 
\begin{equation}
p_i=\frac{|z_i|^2}{\sum_{j=0}^3|z_j|^2}
\end{equation}
for $i=0,...,3$ and
\begin{equation}
\varphi=arg(z_1z_2\overline{z_3z_0}).
\end{equation}
The union
\begin{equation}
\left\{p_1=p_2, p_3=p_0, \varphi=\pm \frac{\pi}{2}\right\}
\end{equation}
of two embedded Lagrangian $\mathbb{RP}^3$s is superheavy.
\end{thm}

It turns out that, in the case of Theorem \ref{Klein_bottle}, the case $\alpha=\pi$ of Theorem \ref{RP2s}, and the case of Theorem \ref{RP3s}, the monotone Lagrangian in question can be realised as the skeleton of the complement of an (algebraic) divisor in a multiple of the anticanonical class, equipped with the usual compactification one-form.

In \cite{BSV}, Borman, Sheridan and Varolgunes construct a special family of Hamiltonians adapted to an anticanonical normal crossings symplectic divisor in a monotone symplectic manifold (in the algebraic setting, an anticanonical normal crossings divisor in a projective variety), which they use, among other things, to show that, under certain numerical conditions on the divisor, the skeleton of the complement is SH-full, a strong condition which implies non-displaceability (in \cite{MSV}, SH-full sets are characterised as those which intersect all closed heavy sets). 

Usually, when trying to understand the symplectic geometry of a divisor complement in a projective variety, the normal crossings condition is no hindrance, because we can blow up along the divisor without changing the geometry of the complement. However, displaceability of the Lagrangian skeleton is a relative property of the skeleton inside the ambient projective variety. Indeed, relative symplectic cohomology, the invariant constructed by Varolgunes, associated to any compact subset of a compact symplectic manifold, and used in \cite{BSV} to prove the rigidity of the skeleton, is not a local invariant in this sense.

Because the anticanonical divisors in our two examples are not normal crossings, we cannot directly apply the result in \cite{BSV}. However, in cases where the divisor has nice enough singularities, including our two examples, we can construct families of Hamiltonians with similar properties to those constructed in \cite{BSV}, which we could use to prove that the skeleta are SH-full. In fact, we will use this construction to prove that the skeleta are superheavy (a stronger condition, according to Corollary 3.10 of \cite{MSV}), using similar arguments to those of Entov and Polterovich in \cite{EP}.

More generally, as in \cite{BSV}, even when we cannot show that the skeleton is $SH$-full, we can find a SH-full (in fact, superheavy) neighbourhood of which the skeleton is a retract, contained in the divisor complement, the volume of which depends on numerical information from the divisor.

For example, we can find such neighbourhoods for the following one-dimensional cell complexes in $\mathbb{CP}^2$.
\begin{thm}
\label{arcs_nbhd}
Consider $\mathbb{CP}^2$ as a symplectic manifold, equipped with the Fubini-Study symplectic form, normalised so a line has area $1$. 

Let $[z_0:z_1:z_2]$ be the standard coordinates on $\mathbb{CP}^2$, and $(p_i,\theta_i)$ be action-angle coordinates as in Theorem \ref{RP2s}. 

The union
\begin{equation}
\bigcup_{k=1}^n\left\{z_0=0, \theta_1-\theta_2=\frac{2\pi}{n}\left(k+\frac{1}{2}\right)\right\},
\label{n_circles}
\end{equation}
of $n$ arcs in $\{z_0=0\}\subset\mathbb{CP}^2$ is a retract of a neighbourhood $U\subset\mathbb{CP}^2$ which is superheavy, and occupies $\frac{1}{9}$ of the volume of $\mathbb{CP}^2$.
\end{thm}

We can also obtain such a neighbourhood of the Chiang Lagrangian in $\mathbb{CP}^3$ (constructed in \cite{Chiang}).
\begin{thm}
\label{Chiang_nbhd}
The Chiang Lagrangian is a retract of a neighbourhood $U\subset\mathbb{CP}^3$ which is superheavy, and occupies $\frac{1}{216}$ of the volume of $\mathbb{CP}^3$.
\end{thm}
The Chiang Lagrangian itself is known to have non-zero Floer cohomology in characteristic $5$ (see \cite{EL}), but this result holds in any characteristic. 

\subsection{Key Idea}
The main original part of this paper is the construction of a smoothing family of contact hypersurfaces for certain divisors, with prescribed Reeb dynamics. Consider, for example, the divisor $\{y=0\}\cup \{y=x^2\}\subset\mathbb{C}^2$, where $(x,y)$ are coordinates on $\mathbb{C}^2$. 

The irreducible components $\{y=0\}$ and $\{y=x^2\}$ are both embedded, and we can find standard symplectic tubular neighbourhoods of these curves. For example, the curve $\{y=0\}$ is embedded as the zero-section of the standard product bundle, which carries a Hamiltonian fibrewise $U(1)$-action, generated by $\frac{1}{2}|y|^2$. The Reeb dynamics on the level sets of $|y|$ all come from this circle action. Symplectically, we could do a similar construction for $\{y=x^2\}$, or indeed any symplectic submanifold of codimension 2. 

To turn these circle bundles over the two curves into a single hypersurface enclosing the union of the two curves, we could do a surgery near the intersection point $(0,0)$, but it is not clear how the Reeb dynamics interact on the intersection (we would like the circle actions to commute, but it may not be possible to ensure this).

There is a Hamiltonian $\mathbb{R}/\mathbb{Z}$-action $t\mapsto diag(e^{4\pi it},e^{2\pi it})$ generated by $\frac{1}{2}|x|^2+|y|^2$. The level sets of this Hamiltonian give us a family of embedded ellipsoids enclosing the point $(0,0)$. Because the circle action preserves the divisor, we can ensure that our tubular neighbourhoods for each curve are equivariant, so the Reeb flow on each circle bundle commutes with the Reeb flow on the ellipsoids.

Figure \ref{symp_reduction_tubular_nbhds} shows the symplectic reduction of one of the ellipsoids, which is a sphere with two singular points. The intersection of the divisor components with the ellipsoid are points on the sphere, and the intersections of the circle bundles with the ellipsoid are closed curves enclosing these points. We ensure that the circle bundles are small enough that these curves are disjoint.

Our hypersurface would then be obtained from attaching both circle bundles separately to the ellipsoid via a surgery. At any point, the Reeb flow on this hypersurface comes from one of the three different $\mathbb{R}/\mathbb{Z}$-actions, or from the $\mathbb{R}^2/\mathbb{Z}^2$-action coming from the intersection of the ellipsoid and one of the circle bundles. At no point will the circle actions from both circle bundles contribute simultaneously (see Figure \ref{ellipsoid_tubular_nbhds}), so it does not matter whether they commute.

\begin{figure}[!h]
\begin{minipage}[c]{0.45\linewidth}
\centering
\caption{Intersection of tubular neighbourhoods of divisor components are contained in ellipsoid.}
\begin{tikzpicture}
\label{ellipsoid_tubular_nbhds}
\filldraw[red!20] (0,0.5)--(5,0.5)--(5,-0.5)--(0,-0.5)--cycle;
\filldraw[blue!20] (0,3).. controls (1.8,-0.5) ..(2.5,-0.5) .. controls (3.2,-0.5) .. (5,3)--(5,5).. controls (2.8,0.5) ..(2.5,0.5) .. controls (2.2,0.5) .. (0,5)--cycle;
\filldraw[gray!30] (2.5,0) ellipse (2 and 1);

\node[text=blue,] at (4.4,4) {$y=x^2$};
\node[text=red] at (4.4,0.2) {$y=0$};

\draw[blue] (0,4).. controls (2,0) ..(2.5,0) .. controls (3,0) .. (5,4);
\draw[red] (0,0)--(5,0);
\end{tikzpicture}
\end{minipage}\hfill
\begin{minipage}[c]{0.45\linewidth}
\centering
\caption{Symplectic reduction of ellipsoid}
\begin{tikzpicture}
\label{symp_reduction_tubular_nbhds}
\draw (2.5,2.5) circle (2.5);
\draw (2.5,5) circle (4pt);
\filldraw[red] (2.5,0) circle (4pt);
\filldraw[blue] (5,2.5) circle (4pt);
\draw[dashed] (0,2.5)--(5,2.5);
\draw [blue] (4.5,4)--(4.5,1);
\draw [red] (1,0.5)--(4,0.5);
\end{tikzpicture}
\end{minipage}
\end{figure}

We will generalise this approach to cases where the non-smooth parts of an algebraic divisor are cut out by quasihomogeneous polynomials, so we can use the $\mathbb{C}^*$-action to obtain a Hamiltonian $\mathbb{R}/\mathbb{Z}$-action near the singular strata, which we can use to control the Reeb dynamics.

\subsection{Background}
In \cite{TMZ}, McLean, Tehrani and Zinger introduce the notion of a simple crossings (SC) symplectic subvariety in a symplectic manifold $(M,\omega)$: a collection $D=\bigsqcup_i D_i$ of transversely intersecting symplectic submanifolds of real codimension 2, with compatible orientations. The latter condition allows $M$  to be equipped with a compatible almost-complex structure such that the submanifolds $D_i$ are almost-complex submanifolds.

A system of commuting Hamiltonians for $D$, as defined by McLean in \cite{McL_growth} consists of a function 
\begin{equation}
r_i:U_i\rightarrow [0,R)
\end{equation}
defined on a neighbourhood $U_i$ of $D_i$ with $r_i^{-1}(0)=D_i$ for each $i$, which generates a Hamiltonian $\mathbb{R}/\mathbb{Z}$-action fixing $D_i$, such that $\{r_i,r_j\}=0$ on $U_i\cap U_j$. McLean shows that a system of commuting Hamiltonians always exists if the components $D_i$ intersect orthogonally.

Let $X=M\setminus D$, $\theta\in \Omega^1(X)$ be a primitive for $\omega$, and $Z$ the associated Liouville vector field.

The primitive $\theta$ is said to be adapted to a system of commuting Hamiltonians if $\theta|_{U_i}$ is invariant under the $\mathbb{R}/\mathbb{Z}$-action on $U_i$.

Let $\rho^0:M\rightarrow \mathbb{R}$ be the unique continuous function such that $\rho^0|_D=1$ and along each flowline of $Z$, $d\rho^0(Z)=\rho^0$.

Given a system of commuting Hamiltonians for $D$, such that $\theta$ is adapted and
\begin{equation}
2c_1(TM)=\sum_i \lambda_i PD(D_i)\in H^2(M;\mathbb{R})
\end{equation}
and
\begin{equation}
\label{lambdas}
\int_u\omega-\int_{\partial u}\theta=\kappa\sum_i \lambda_i u\cdot D_i,
\end{equation}
for $u\in H_2(M,X)$, where $\lambda_i$ is a positive rational number for each $i$,  Borman, Sheridan and Varolgunes construct a smooth family $\rho^R$ of approximations to $\rho^0$.

In the case where $(M,\omega)$ is a monotone symplectic manifold, they show that the action and Conley-Zehnder index of 1-periodic orbits of $h\circ \rho^R$ (for some auxiliary function $h$) obey certain estimates, which they use, for $\sigma(R)<\sigma_1<\sigma_2<1$, to construct an isomorphism
\begin{equation}
\label{relSH_isom}
SH_M^*((\rho^R)^{-1}[\sigma_1,\infty);\Lambda)\rightarrow SH_M^*((\rho^R)^{-1}[\sigma_2,\infty);\Lambda),
\end{equation}
where $\sigma(R)$ is a certain continuous function of $R$. 

This implies that the set $K_{crit}=(\rho^0)^{-1}[0,\sigma_{crit}]$ is $SH$-full, where $\sigma_{crit}=\sigma(0)$, because symplectic divisors are stably displaceable, and therefore some neighbourhood $(\rho^R)^{-1}[\sigma_2,1]$ of $D$ is $SH$-invisible, and by the isomorphism, so are sets of the form $(\rho^R)^{-1}[\sigma_1,1]$ for $\sigma(R)<\sigma_1<\sigma_2$, which exhaust $(\rho^0)^{-1}(\sigma_{crit},1]$.

In particular, if $\lambda_i\leq 2$ for all $i$ (Hypothesis A of \cite{BSV}), then $\sigma_{crit}=0$ and the Lagrangian skeleton $(\rho^0)^{-1}(0)$ of $(X,\theta)$ is $SH$-full.

\subsection{Outline}
In this paper, we will attempt to apply similar techniques to a larger class of `symplectic divisors' $D$.

In Section \ref{s2}, we define the notion of a stratified symplectic subvariety in a symplectic manifold, generalising SC symplectic divisors.

In Section \ref{s3}, we give a definition of a system of commuting Hamiltonians for a stratified symplectic subvariety analogous to the case of a SC divisor. Unlike the SC case, we do not have a general criterion for the existence of a system of commuting Hamiltonians like orthogonality of the divisor components, but we give a partial result which will be sufficient for all our examples in Proposition \ref{hypersurf_ham}.

In Section \ref{s4}, we define what it means for a Liouville primitive on the complement of a stratified symplectic subvariety to be adapted to a system of commuting Hamiltonians, and give a criterion under which a given primitive is equivalent to an adapted one. This implies that if a given primitive satisfies this criterion, there exists an adapted primitive with the same Lagrangian skeleton.

In Section \ref{s5}, following Borman, Sheridan and Varolgunes' construction, we use a system of commuting Hamiltonians and adapted Liouville primitive to construct a smoothing family $\rho^R$ for $\rho^0$, in the case of a stratified symplectic subvariety. We obtain estimates for action and Conley-Zehnder index of orbits of $h\circ\rho^R$. In the case of a monotone symplectic manifold, we then use these computations, together with techniques of Entov and Polterovich, to show that $K_{crit}$ is superheavy.
 
We will see that, in general, $\sigma_{crit}$ depends on data from each stratum of the symplectic subvariety, not only, as in the SC case, the cohomology classes represented by the irreducible components. Consequently, the appropriate generalisation of Hypothesis A will, in particular, be sensitive to the local topology of singularities of the symplectic subvariety (see Corollary \ref{superheavy_skeleton}).

In Section \ref{s6}, we specialise to the algebraic setting, showing in particular that a system of commuting Hamiltonians always exists when $M$ is a projective Fano surface and $D$ an anticanonical divisor locally cut out by quasihomogeneous equations (Corollary \ref{qhsch}). We use this to establish a simple algebraic criterion for the skeleton of the affine complement of such a divisor to be superheavy.

Finally, we apply these results to some skeleta of complements of anticanonical $\mathbb{Q}$-divisors in Fano varieties, first constructing a system of commuting Hamiltonians, then calculating $\sigma_{crit}$.

By this method, we first prove Theorem \ref{Klein_bottle} by realising $K$ as the skeleton of the complement of an anticanonical divisor cut out by a quasihomogeneous polynomial, applying Corollary \ref{qhsch}, and showing $\sigma_{crit}=0$. We then prove Theorem \ref{RP2s} by the same method (though in general the unions of Lagrangian $\mathbb{RP}^2$s in question are only Hamiltonian isotopic to the skeleta). We then use an extended version of Corollary \ref{qhsch} slightly, to prove Theorem \ref{RP3s}

We end with two examples where $\sigma_{crit}\neq 0$. First, we use Corollary \ref{qhsch} to prove Theorem \ref{arcs_nbhd}.

Lastly, we prove Theorem \ref{Chiang_nbhd}, by realising the Chiang Lagrangian as the skeleton of the complement of an irreducible, singular anticanonical surface, and using the existence of a $SU(2)$-action on $\mathbb{CP}^3$, which preserves the divisor, to construct a system of commuting Hamiltonians. A smooth anticanonical surface would satisfy hypothesis A, but here $\sigma_{crit}\neq 0$ due to a contribution from the geometry of the singular locus of the surface.

\subsection{Acknowledgements}
Many thanks to Yuhan Sun for helpful comments about the relationship between SH-fullness and superheaviness, and for suggesting using the use of arguments from \cite{EP} to prove superheaviness rather than just SH-fullness.

\section{Stratified Symplectic Subvarieties}
\label{s2}
\subsection{Stratified Posets}
\begin{defn}[Stratified Poset]
A stratified poset $V$ of height $n$ consists of a finite vertex set $V$, together with a partial order $\leq$ on $V$, and a decomposition
\begin{equation}
V=\bigsqcup_{k=0}^nV_k
\end{equation}
such that if $u\in V_k$ and $v\in V_l$ and $u < v$, then $k < l$.
\end{defn}

\begin{defn}[Stratified Subposet]
A stratified subposet of a stratified poset $V$ of height $n$ is a subset $W\subset V$ such that if $u\leq v$ and $v\in W$ then $u\in W$.

The set $W$, together with the restriction of the partial order $\leq$ and the decomposition $W_k=V_k\cap W$ for $k=0,...,n$ is naturally a stratified poset of height $\leq n$.
\end{defn}

\begin{defn}
From a stratified poset $V$ of height $n$, for $k=0,...,n$, we obtain a stratified subposet of height $k$, 
\begin{equation}
V_{\leq k}=\bigsqcup_{l=0}^kV_l.
\end{equation}
\end{defn}

\begin{defn}
From a stratified poset $V$ and a vertex $v\in V_k$, we obtain a subposet of height $k$, 
\begin{equation}
V_{\leq v}=\{u\in V|u\leq v\}.
\end{equation}
\end{defn}

\begin{defn}
From a stratified poset $V$ and a vertex $v\in V_k$, we obtain a subposet of height $k-1$, also a subposet of $V_{\leq v}$, 
\begin{equation}
V_{<v}=\{u\in V|u<v\}.
\end{equation}
\end{defn}

\subsection{Stratified Symplectic Subvarieties}

\begin{defn}[Stratified Symplectic Subvariety]
A stratified symplectic subvariety of dimension $2n$ in a closed symplectic manifold $(M,\omega)$ is a stratified poset $V$ of height $n$, together with a collection of disjoint subsets
\begin{equation}
\mathring{D}_v\subset M
\end{equation}
for $v\in V$, such that whenever $v\in V_k$, $\mathring{D}_v$ is an embedded (possibly open) symplectic submanifold of dimension $2k$, and the closure of $\mathring{D}_v\subset M$ is 
\begin{equation}
\bigcup_{u\leq v}\mathring{D}_u\subset M.
\end{equation}

We denote a stratified symplectic subvariety solely by the indexing poset $V$.
\end{defn}

\begin{defn}[Subvariety of a Stratified Symplectic Subvariety]
Given a stratified symplectic subvariety $V$ in $M$, any subposet $W\subset V$ is naturally a stratified symplectic subvariety in $M$. We call the subposet together with the natural embedded stratified symplectic subvariety data a subvariety of $V$.

We denote a subvariety of $V$ solely by the indexing subposet $W$.
\end{defn}

\begin{defn}
Given a stratified symplectic subvariety $V$ in $M$, the realisation $|V|$ of $V$ is given by 
\begin{equation}
|V|=\bigcup_{v\in V}\mathring{D}_v\subset M.
\end{equation}
\end{defn}

\begin{ex}
\label{cubic_Example}
Let $M=\mathbb{CP}^2$, equipped with the Fubini-Study symplectic form. Let $[z_0:z_1:z_2]$ be homogeneous coordinates.

We can construct a stratified symplectic subvariety of $M$ associated to the stratified poset $V=\{0,1\}$ with $0<1$, and $V_i=\{i\}$ for $i=0,1$, such that 
\begin{equation}
|V|=\{z_0z_1^2=z_2^3\}\subset M
\end{equation}
by letting 
\begin{equation}
\mathring{D}_1=\{z_0z_1^2=z_2^3\}\setminus \{[1:0:0]\}
\end{equation}
and 
\begin{equation}
\mathring{D}_0=[1:0:0].
\end{equation}

We have that $\mathring{D}_1=\{y=x^3\}$ where \begin{equation}
(x,y)=\left(\frac{z_2}{z_1},\frac{z_0}{z_1}\right)
\end{equation}
are affine coordinates on $\mathbb{CP}^2\setminus \{z_1=0\}\cong \mathbb{C}^2$, so $\mathring{D}_1$ is indeed an embedded, open symplectic submanifold of $M$, of real dimension $2$, whereas the point $\mathring{D}_0$ is trivially an embedded, zero-dimensional symplectic submanifold.
\end{ex}

\begin{ex}[Simple Crossings Symplectic Divisors]
\label{NC_case}
Suppose we have a SC symplectic divisor in the sense of Definition 2.1 of \cite{TMZ}, i.e. we have embedded closed symplectic submanifolds  $D_1,...,D_k\subset M$ of real dimension $2(n-1)$ where $M$ is $2n$-dimensional, which intersect transversely, with compatible orientations. 

We can define a stratified poset of height $n-1$,
\begin{equation}
V=\{v\subset \{1,...,k\}| \bigcap_{i\in v}D_i\neq \emptyset\}\subset \mathcal{P}(\{1,...,k\}).
\end{equation}
letting $u\leq v$ if $u\subset v$, and
\begin{equation}
V_d=\{v\in V| |v|=n-d\}.
\end{equation}

We can give $V$ the structure of a stratified symplectic subvariety by letting, for $v\subset \{1,...,k\}$,
\begin{equation}
\mathring{D}_v=\bigcap_{i\in v}D_i\cap \bigcap_{i\notin v}(M\setminus D_i)
\end{equation}
which is a symplectic submanifold of dimension $2(n-|v|)$.
\end{ex}

\begin{lem}
If $V$ is a stratified symplectic subvariety of height $\leq n-1$ in a $2n$-dimensional symplectic manifold $M$, then $|V|$ is a partly stratified symplectic subset of $M$ in the sense of Definition 6.18 of \cite{McL_birational}, i.e. $|V|$ is equal to a union of disjoint subsets $S_1,...,S_l$ such that for each $j$, $\cup_{i\leq j}S_i$ is compact, and $S_j$ is a proper symplectic submanifold of $M\setminus \cup_{i<j}S_i$ of codimension $\geq 2$.
\end{lem}

\begin{proof}
Choose a total order on $V$ extending the partial order. This gives a bijective  map of posets $f:V\rightarrow \{1,...,l\}$ for some $l$.

Take $S_j=\mathring{D}_{f^{-1}(v)}$. The subsets $S_j$ satisfy the required properties.
\end{proof}

\begin{cor}
\label{stably_displaceable}
If $V$ is a stratified symplectic subvariety of dimension $\leq 2(n-1)$ in a $2n$-dimensional symplectic manifold $M$, then $|V|$ is a stably displaceable subset of $M$.
\end{cor}
\begin{proof}
Follows immediately from Proposition 6.20 of \cite{McL_birational}.
\end{proof}
\subsection{Stratum Neighbourhoods}

\begin{defn}[Stratum Neighbourhood]
Let $V$ be a symplectic subvariety $V$ in $M$. A stratum neighbourhood for a vertex $v\in V$ is a neighbourhood $U_v$ of the embedded submanifold $\mathring{D}_v$ in $M$.
\end{defn}

\begin{defn}[Intersection Neighbourhood]
Given a stratified symplectic subvariety $V$ in $M$, and vertices $u\leq v\in V$, an intersection neighbourhood for the pair $u\leq v$ is, for some neighbourhood $W$ of $\mathring{D}_u$, a neighbourhood of $W\cap \mathring{D}_v$. 
\end{defn}

\begin{ex}
Returning to the case of the cubic curve in Example \ref{cubic_Example}, figure \ref{cubic_vertex_neighbourhoods} illustrates   a choice of stratum neighbourhoods $U_0$, $U_1$ for vertices $0$ and $1$ respectively, i.e. neighbourhoods of the singular point and the smooth locus of the curve, respectively.

Figure \ref{cubic_intersection_neighbourhood} illustrates a choice of intersection neighbourhood $I$ for the pair $0<1$. $W$ is a neigbhourhood of the singular point, and $I$ is a neighbourhood of the intersection of the smooth locus of the curve with $W$.

\begin{figure}
\caption{Stratum neighbourhoods for components of the cubic curve in Example \ref{cubic_Example} (real locus in $\mathbb{CP}^2\setminus \{z_0=0\}$ shown)}
\centering
\begin{tikzpicture}[scale=0.8]
\label{cubic_vertex_neighbourhoods}
\fill[fill=gray!20] (5,1) circle (1);

\fill[gray!40] (1,4).. controls (1.5,4) and (5,2.5) .. (5,1) .. controls (5,1.5) and (1.5,2.5) .. (1,2.5) -- cycle;

\fill[gray!40] (9,4).. controls (8.5,4) and (5,2.5) .. (5,1) .. controls (5,1.5) and (8.5,2.5) .. (9,2.5) -- cycle;

\draw[dashed] (5,1) circle (1);

\draw (1,3).. controls (1.5,3) and (5,2) .. (5,1);
\draw (9,3).. controls (8.5,3) and (5,2) .. (5,1);

\node at (1,3) [anchor=north west]{$\mathring{D}_1$};

\filldraw[black] (5,1) circle (0.06) node[anchor=west]{$\mathring{D}_0$};

\node at (2,3.2) {$U_1$};

\node at (5,0.5) {$U_0$};

\end{tikzpicture}
\end{figure}
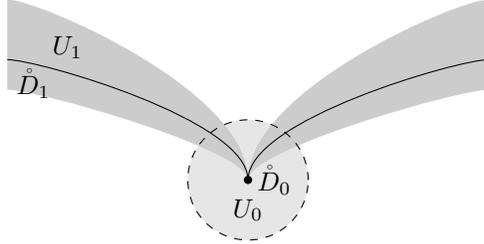

\begin{figure}
\caption{Intersection neighbourhood for components of the cubic curve in Example \ref{cubic_Example}.}
\centering
\begin{tikzpicture}[scale=0.8]
\label{cubic_intersection_neighbourhood}

\draw[fill=gray!30, dashed] (5,1.9) ellipse (1.1 and 1);

\draw[fill=gray!60, dashed] (5,1).. controls (5,3) and (2.5,2.5) .. (5,1) -- cycle;

\draw[fill=gray!60, dashed] (5,1).. controls (5,3) and (7.5,2.5) .. (5,1) -- cycle;

\draw (1,3).. controls (1.5,3) and (5,2) .. (5,1);
\draw (9,3).. controls (8.5,3) and (5,2) .. (5,1);

\node at (4.4,2) {$I$};
\node at (5,2.2) {$W$};
\node at (1,3) [anchor=north west]{$\mathring{D}_1$};

\filldraw[black] (5,1) circle (0.06) node[anchor=north]{$\mathring{D}_0$};
\end{tikzpicture}
\end{figure}
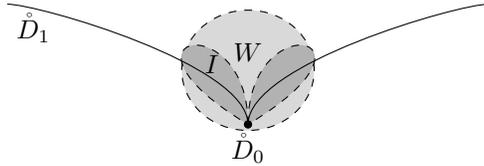
\end{ex}

\begin{lem}
\label{tidy_1}
Given a stratified symplectic subvariety $V$ in $M$, let $U_u,U_v$ be stratum neighbourhoods for vertices $u\leq v\in V$ respectively.

Then $U_u\cap U_v$ is an intersection neigbourhood for $u\leq v$.
\end{lem}

\begin{proof}
$U_u$ is a neighbourhood of $\mathring{D}_u$, and $U_u\cap U_v$ is a neighbourhood of $U_u\cap \mathring{D}_v$.
\end{proof}

\begin{lem}
\label{shrink_vertex_nbhds}
Given a stratified symplectic subvariety $V$ in $M$, let $u\leq v\in V$ be vertices, and $I$ an intersection neighbourhood for $u\leq v$.

Then there exist stratum neighbourhoods $U_u,U_v$ for $u,v$ respectively, such that 
\begin{equation}
U_u\cap U_v=I.
\end{equation} 
\end{lem}

\begin{proof}
We have that $I$ is a neighbourhood of $W\cap \mathring{D}_u$ for some neighbourhood $W$ of $\mathring{D}_u$. 

The sets $\mathring{D}_u$ and $\mathring{D}_v\setminus W$ have disjoint closures, so we may choose disjoint neighbourhoods $Y$ and $Z$ of $\mathring{D}_u$ and $\mathring{D}_v\setminus W$, respectively.  

Let $U_u=Y\cup I$ and $U_v=Z\cup I$. As required, $U_u\cap U_v=I$, $U_u$ is a neighbourhood for $\mathring{D}_u$ and $U_v$ is a neighbourhood for $\mathring{D}_v$.
\end{proof}

\begin{ex}
Again in the case of the cubic curve in Example \ref{cubic_Example}, figure \ref{shrink_nbhds_Example} illustrates the proof of Lemma \ref{shrink_vertex_nbhds} applied to the stratum neighbourhood $I$ shown in figure \ref{cubic_intersection_neighbourhood}.

$Y$ is a neighbourhood of the singular point, contained in $W$, and $Z$ is a neighbourhood of the complement of $W$ in the curve, so $Y\cup I$ and $Z\cup I$ are neighbourhoods of the singular point and the smooth locus of the curve, respectively.

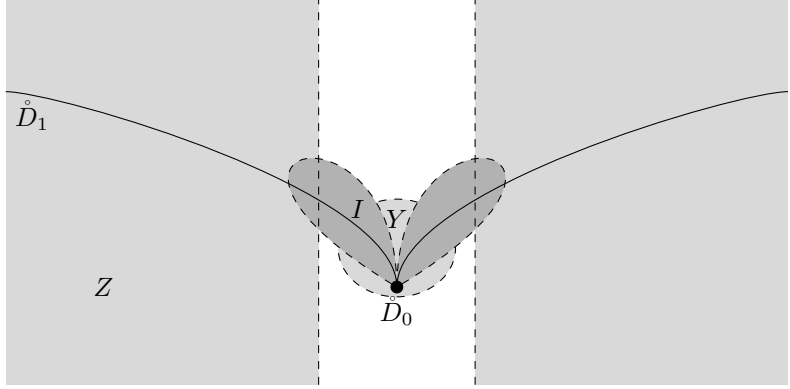
\begin{figure}
\caption{Proof of Lemma \ref{shrink_vertex_nbhds} in the case of the cubic curve in Example \ref{cubic_Example}.}
\centering
\begin{tikzpicture}[scale=1.3]
\label{shrink_nbhds_Example}
\draw[fill=gray!30, draw=none] (1,0) rectangle (4.2,4);
\draw[fill=gray!30, draw=none] (9,0) rectangle (5.8,4);
\draw[fill=gray!30, dashed] (5,1.4) ellipse (0.6 and 0.5);

\draw[fill=gray!60, dashed] (5,1).. controls (5,3) and (2.5,2.5) .. (5,1) -- cycle;

\draw[fill=gray!60, dashed] (5,1).. controls (5,3) and (7.5,2.5) .. (5,1) -- cycle;

\draw (1,3).. controls (1.5,3) and (5,2) .. (5,1);
\draw (9,3).. controls (8.5,3) and (5,2) .. (5,1);

\draw[dashed, fill=none] (4.2,0)--(4.2,4);
\draw[dashed, fill=none] (5.8,0)--(5.8,4);

\node at (4.6,1.8) {$I$};
\node at (5,1.7) {$Y$};
\node at (2,1) {$Z$};
\node at (1,3) [anchor=north west]{$\mathring{D}_1$};

\filldraw[black] (5,1) circle (0.06) node[anchor=north]{$\mathring{D}_0$};
\end{tikzpicture}
\end{figure}
\end{ex}

\section{Systems of Commuting Hamiltonians}
\label{s3}

\subsection{Radial Hamiltonians}

Throughout this section, let $V$ be a stratified symplectic subvariety in a $2n$-dimensional symplectic manifold $(M,\omega)$.

\begin{defn}[Radial Hamiltonian]
A radial Hamiltonian for a vertex $v\in V$ is a stratum neighbourhood $U$ for $v$, together with a a smooth function
\begin{equation}
r:U\rightarrow \mathbb{R}_{\geq 0}
\end{equation}
such that
\begin{itemize}
\item $r^{-1}(0)=\left(\bigcup_{u\leq v}\mathring{D}_u\right)\cap U$,
\item $r$ generates a Hamiltonian $\mathbb{R}/\mathbb{Z}$-action on $U$, denoted $t\cdot x=\phi_r^t(x)$,
\item the fixed points of the action are exactly $r^{-1}(0)$,
\item for any $v\leq w$, the action preserves $U\cap\mathring{D}_w$.
\end{itemize}
\end{defn}

\begin{defn}[Weight of a Radial Hamiltonian]
Given a radial Hamiltonian $r:U\rightarrow \mathbb{R}_{\geq 0}$ for a vertex $v\in V_k$, at any point $x\in \mathring{D}_v$, the Hamiltonian $\mathbb{R}/\mathbb{Z}$-action generated by $r$ induces an isotropy representation of $S^1$ on $(T_x\mathring{D}_v)^\omega\cong \mathbb{C}^{n-k}$, which has the form
\begin{equation}
e^{it}\mapsto diag(e^{a_{1}it},...,e^{a_{n-k}it})
\end{equation}
with respect to some choice of symplectic basis, for some integers $a_{i}\in\mathbb{Z}$ for $i=1,...,n-k$.

Because $\mathring{D}_v$ is a local minimum for $r$, we have $a_i>0$ for all $i$ (see Theorem 4.7 of \cite{GS}).

We define the total weight $w\in\mathbb{N}$ of $r$ to be
\begin{equation}
w=\sum_{i=1}^{n-k}a_{i}.
\end{equation}
This is independent of the point $x$ by continuity.
\end{defn}

\begin{ex}
\label{cubic_circle_action}
We return once again to the cubic curve from Example \ref{cubic_Example}.

The Hamiltonian 
\begin{equation}
r_0=\frac{3|z_1|^2+2|z_2|^2}{|z_0|^2+|z_1|^2+|z_2|^2}
\end{equation}
generates a $\mathbb{R}/\mathbb{Z}$-action on $\mathbb{CP}^2$ of the form
\begin{equation}
t\cdot [z_0:z_1:z_2]=[z_0:e^{6\pi i t}z_1:e^{4\pi i t}z_2].
\end{equation}

This action preserves the curve $z_0z_1^2=z_2^3$, and the fixed points are exactly $[1:0:0]=r_0^{-1}(0)$, so $r_0$ is a radial Hamiltonian for the vertex $0$.

The isotropy representation of the action is 
$t\mapsto diag(e^{6\pi i t},e^{4\pi i t})$, so we have that the weight $w_0$ of $r_0$ is given by $w_0=3+2=5$.
\end{ex}

\begin{defn}(Invariant Subsets)
Let $V$ be a stratified symplectic subvariety in $M$, and $r:U\rightarrow \mathbb{R}_{\geq 0}$ be a radial Hamiltonian for $v\in V$.

A subset $Y\subset U$ is $r$-invariant if the Hamiltonian $\mathbb{R}/\mathbb{Z}$-action generated by $r$ restricts to an action on $Y\cap U$.

Note that if $T\subset U$ is $r$-invariant and a stratum neighbourhood for $v$, then $r|_T$ is also a radial Hamiltonian for $v$.
\end{defn}

\begin{defn}(Invariant Tensors)
Let $V$ be a stratified symplectic subvariety in $M$, and $r:U\rightarrow \mathbb{R}_{\geq 0}$ be a radial Hamiltonian for $v\in V$, and $Y\subset U$ an $r$-invariant open subset.

A tensor $s$ on $Y$ is $r$-invariant if $(\phi_r^t)^*s=s$ for all $t\in \mathbb{R}/\mathbb{Z}$.
\end{defn}

\subsection{Systems of Commuting Hamiltonians}
\begin{defn}[Commuting Radial Hamiltonians]
Let $v$ be a stratified symplectic subvariety in $M$, and $r_u:U_u\rightarrow \mathbb{R}_{\geq 0}$ and $r_v:U_v\rightarrow \mathbb{R}_{\geq 0}$ be radial Hamiltonians for vertices $u\leq v\in V$, respectively. 

We say that $r_u$ commutes with $r_v$ if for some intersection neighbourhood $I$ for the pair $u\leq v$, 
\begin{equation}
\{r_u|_I,r_v|_I\}=0.
\end{equation}
\end{defn}

\begin{defn}[System of Commuting Hamiltonians]
\label{sch_def}
Let $V$ be a stratified symplectic subvariety in $M$. A system of commuting Hamiltonians for $V$ is a choice of radial Hamiltonian $r_v:U_v\rightarrow \mathbb{R}_{\geq 0}$ for each vertex $v\in V$, such that, if $u\leq v$, $r_u$ commutes with $r_v$.

For any subset $I\subset V$, we denote
\begin{equation}
U_I=\bigcap_{v\in I}U_v
\end{equation}
and
\begin{equation}
r_I=(r_v:v\in I):U_I\rightarrow \mathbb{R}_{\geq 0}^I.
\end{equation}

We say that a subset $Y\subset U_I$ is $r_I$-invariant if it is $r_v$-invariant for $v\in I$, and a tensor $s$ on a $r_I$-invariant subset $Y$ is $r_I$-invariant if $s$ is $r_v$-invariant for $v\in I$.
\end{defn}

\begin{ex}[Normal Crossings Symplectic Divisors]
Let $V$ be a stratified symplectic divisor associated to a SC symplectic divisor as in Example \ref{NC_case}. 

Suppose the submanifolds $D_i\subset M$ intersect symplectically orthogonally, i.e. the SC divisor is normal crossings.

By Lemma 5.14 of \cite{McL_growth}, there exists a system of commuting Hamiltonians 
\begin{equation}
r_i:U_i\rightarrow [0,R)
\end{equation}
for $D$ in the sense of \cite{McL_growth} and \cite{BSV}. We can use this to construct a system of commuting Hamiltonians for $V$ in the sense of Definition \ref{sch_def} by letting
\begin{equation}
r_v=\sum_{i\in v}r_i:\bigcap_{i\in v}U_i\rightarrow \mathbb{R}_{\geq 0}
\end{equation}
for $v\in V$.

Because
\begin{equation}
T_xD_v^\omega=\bigoplus_{i\in v}T_xD_v^\omega,
\end{equation}
we have that the weight $w_v$ of $r_v$ is $|v|$.
\end{ex}

\begin{ex}
To extend our radial Hamiltonian $r_0$ described in Example \ref{cubic_circle_action} to a system of commuting Hamiltonians for the cubic curve, we would require a radial Hamiltonian $r_1:U_1\rightarrow \mathbb{R}_{\geq 0}$ for the vertex $1$, which commutes with $r_0$, i.e. there exists an intersection neighbourhood $I$, like the one shown in figure \ref{cubic_intersection_neighbourhood}, such that $\{r_0|_I,r_1|_I\}=0$, so that $(r_0,r_1)$ gives us a (possibly singular) Lagrangian torus fibration on $I$. 

The following parts of this section build towards a construction of a Hamiltonian like $r_1$ in general, when, as in this case, we have a subvariety $V$ of dimension $2(n-1)$, and commuting radial Hamiltonians for each $v\in V_{\leq n-2}$, and seek to extend these choices to a system of commuting Hamiltonians for $V$ (see Proposition \ref{hypersurf_ham}). 
\end{ex}

\begin{lem}
\label{intersection_commute}
Given a stratified symplectic subvariety $V$ in $M$ and a system of commuting Hamiltonians $r_v:U_v\rightarrow \mathbb{R}_{\geq 0}$ for $v\in V$, we can choose stratum neighbourhoods $T_v\subset U_v$ for $v\in V$ such that
\begin{itemize}
\item for each $u,v\in V$, either $u\leq v$, $v\leq u$, or $T_u\cap T_v=\emptyset$
\item for each $u<v$, $\{r_u|_{T_u\cap T_v},r_v|_{T_u\cap T_v}\}=0$
\item for each $v\in V$, $T_v$ $r_v$-invariant, i.e. $r_v|_{T_v}$ for $v\in V$ is a system of commuting Hamiltonians for $V$.
\end{itemize}
\end{lem}

\begin{proof}
For any pair $u<v$, we have an intersection neighbourhood $I_{u,v}\subset U_u\cap U_v$ such that, on restriction to $I_{u,v}$, $\{r_u,r_v\}=0$. By Lemma \ref{shrink_vertex_nbhds}, we can choose stratum neighbourhoods $T_{u,v}\subset U_u$ and $T_{v,u}\subset U_v$ for $u$ and $v$ respectively such that $T_{u,v}\cap T_{v,u}\subset I$.

For each $u,v\in V$ such that neither $u\leq v$ or $v\leq u$, by the Hausdorff property we can choose stratum neighbourhoods $T_{u,v}$ and $T_{v,u}$ for $u$ and $v$ respectively such that $T_{u,v}\cap T_{v,u}=\emptyset$.

Now, for each $v\in V$, let
\begin{equation}
T_v=\bigcap_{u\neq v}T_{v,u}
\end{equation}
$T_v\subset U_v$ is a stratum neighbourhood for $v$. 

Now we can replace $T_v$ by
\begin{equation}
\bigcap_{t=0}^1\phi_{r_v}^t(T_v)\subset T_v
\end{equation}
which is a stratum neighbourhood for $v$ because $\mathbb{R}/\mathbb{Z}$ is compact and $\mathring{D}_v$ is fixed by the action, so the intersection of all $\mathbb{R}/\mathbb{Z}$-translates of a neighbourhood of $\mathring{D}_v$ remains a neighbourhood.
\end{proof}

\subsection{Invariant Stratum Neighbourhoods and Exhaustion Functions}
\begin{lem}
\label{compact_boxes}
Let $V$ be a stratified symplectic subvariety in $M$, and $r_v:U_v\rightarrow \mathbb{R}_{\geq 0}$ for $v\in V$ smooth functions such that $U_v$ is a stratum neighbourhood for $v$ and $r_v^{-1}(0)=\mathring{D}_v$. For example, $r_v$ for $v\in V$ could be a system of commuting Hamiltonians for $V$.

Then for $v\in V$ there exist continuous, strictly increasing functions 
\begin{equation}
\epsilon_v:[0,1]\rightarrow \mathbb{R}_{\geq 0}
\end{equation}
where $\epsilon_v(0)=0$, such that for $R\in (0,1]$, the set
\begin{equation}
N_v^R=r_v^{-1}[0,\epsilon_v(R)]\cap \left(M\setminus \bigcup_{u<v}r_u^{-1}\left[0,\frac{1}{3}\epsilon_u(R)\right)\right)
\label{N_v^R}
\end{equation}
is compact and contained in the interior of $U_v$.

Furthermore, for any fixed $t>0$, we may require that whenever $u<v$, $\epsilon_v\leq t\epsilon_u$.
\end{lem}
\begin{proof}
We proceed by induction on the poset $V$. If $V=\emptyset$ we are done. Otherwise let $v\in V$ be a maximal element, and suppose functions $\epsilon_u$ exist with all the required properties for $u\neq v$.

For $R\in (0,1]$, the subset
\begin{equation}
\mathring{D}_v\cap \left(M\setminus \bigcup_{u<v}r_u^{-1}\left[0,\frac{1}{3}\epsilon_u(R)\right]\right)\subset M\setminus \bigcup_{u<v}r_u^{-1}\left[0,\frac{1}{3}\epsilon_u(R)\right]
\end{equation}
is a properly embedded submanifold and contained in the interior of $U_v$, and therefore has a compact neighbourhood also contained in the interior of $U_v$. We can choose a not necessarily continuous function $\epsilon_v:(0,1]\rightarrow \mathbb{R}_{>0}$ such that the set
\begin{equation}
r_v^{-1}[0,\tilde{\epsilon}_v(R)]\cap \left(M\setminus \bigcup_{u<v}r_u^{-1}\left[0,\frac{1}{3}\epsilon_u(R)\right]\right)
\end{equation}
is contained in this compact neighbourhood, which implies that the set is itself compact, and contained in the interior of $U_v$.

Finally, we can replace $\epsilon_v$ by a strictly smaller continuous function, if necessary also ensuring $\epsilon_v\leq t\min_{u<v}\epsilon_u$.
\end{proof}

\begin{cor}[Existence of Systems of Invariant Neighbourhoods]
\label{invariant_neighbourhoods}
Let $V$ be a stratified symplectic subvariety  in $M$, and $r_v:U_v\rightarrow \mathbb{R}_{\geq 0}$ for $v\in V$ a system of commuting Hamiltonians for $V$.

Then there exist stratum neighbourhoods $N_v\subset U_v$ for $v\in V$, such that for any $I\subset V$, $N_I$ is $r_I$-invariant, and $N_v\cap X_v\subset X_v$ is closed, where $X_v=M\setminus |V_{<v}|$ for $v\in V$. 
\end{cor}

\begin{proof}
By Lemma \ref{intersection_commute}, we may pass to smaller stratum neighbourhoods and assume that for any $u,v\in V$, $\{r_u|_{U_u\cap U_v},r_v|_{U_u\cap U_v}\}=0$.

Let $\epsilon_v$ and $N_v^R$ for $v\in V$ and $R\in [0,1]$ be as in the statement of Proposition \ref{exhaustion_existence}. For each $v\in V$, let
\begin{equation}
N_v=\bigcup_{0< R\leq 1}N_v^R\subset U_v
\label{N_v}
\end{equation}
This is a stratum neighbourhood for $v$. 

It suffices to show that for any $u,v\in V$, $N_u\cap U_v$ is $r_v$-invariant. This follows from the fact that $N_u^R$ is compact and closed under the flow of $X_{r_v}$ where defined. 
\end{proof}

\begin{prop}
\label{exhaustion_existence}
Let $V$ be a stratified symplectic subvariety in $M$, and $r_v:U_v\rightarrow \mathbb{R}_{\geq 0}$ for $v\in V$ a system of commuting Hamiltonians.

Let $\epsilon_v:[0,1]\rightarrow \mathbb{R}_{\geq 0}$ for $v\in V$ be continuous functions as in the statement of Lemma \ref{compact_boxes}, and $N_v$ for $v\in V$ the stratum neighbourhoods defined using these functions via equations $\eqref{N_v}$ and $\eqref{N_v^R}$, for $v\in V$.

For $I\subset V$, we can define open sets
\begin{equation}
\mathring{U}_I=U_I\cap \left(M\setminus \bigcup_{v\in V\setminus I}N_v\right).
\end{equation}

Then there exist: a smooth, proper function
\begin{equation}
f:X\rightarrow (0,1]
\end{equation}
where $X=M\setminus |V|$, and smooth functions
\begin{equation}
\tilde{f}_I:\mathbb{R}_{> 0}^I\rightarrow (0,1]
\end{equation}

such that
\begin{enumerate}
\item $\tilde{f}_\emptyset=1$,
\item for $v\in I\subset V$, $\partial_v\tilde{f}_I\geq 0$,
\item for $I\subset V$, $d\tilde{f}_I\neq 0$ on $\tilde{f}_I^{-1}(0,1)$,
\item for $v\in I\subset V$ and $r\in \mathbb{R}_{>0}^I$, $r_v> \frac{2}{3}\epsilon_v(\tilde{f}_I(r))$,
\item for $J\subset I\subset V$ and $r\in \mathbb{R}_{>0}^I$, if $r_v\geq \epsilon_v(\tilde{f}_I(x))$ for $v\in I\setminus J$, then $\tilde{f}_J\circ \pi_{I,J}(r)=\tilde{f}_I(r)$,
\item and for $I\subset V$, $f|_{\mathring{U}_I}=\tilde{f}_I\circ r_I$.
\end{enumerate}

\begin{figure}[!h]
\begin{minipage}[c]{0.45\linewidth}
\centering
\caption{Graph of the function $\tilde{g}$, as in the proof of Proposition \ref{exhaustion_existence}.}
\begin{tikzpicture}[scale=0.8]
\label{tildeg}
\draw[->] (-1,0) -- (5,0);
\draw[->] (1,-1) -- (1,5);
\draw[dashed] (-1,2)--(5,2);
\draw[dashed] (3,0)--(3,5);
\draw (1,0).. controls (2,0) .. (3,2)--(4.5,5);
\node at (1,0)[anchor=north east] {0};
\node at (1,5)[anchor=south] {$y$};
\node at (3,0)[anchor=north] {$1$};
\node at (1,2)[anchor=north east] {$1$};
\node at (5,0) [anchor=west]{$x$};
\node at (4,3)[anchor=north west] {$y=\tilde{g}(x)$};
\end{tikzpicture}
\end{minipage}\hfill
\begin{minipage}[c]{0.45\linewidth}
\centering
\caption{A level set of the function $\tilde{f}=\tilde{f}_{\{0,1\}}$, as in the statement and proof of Proposition \ref{exhaustion_existence}. The level set lies above the lower dotted lines because of Condition 4, and is a straight line outside the higher dotted lines because of 5.}
\begin{tikzpicture}[scale=0.7]
\label{tildef}
\draw[->] (-1,0) -- (5,0);
\draw[->] (0,-1) -- (0,5);
\draw[dashed] (0,1)--(5,1);
\draw[dashed] (0,2)--(5,2);
\draw[dashed] (1,0)--(1,5);
\draw[dashed] (2,0)--(2,5);
\draw (1.2,5) --(1.2,2) .. controls (1.2,1.2) .. (2,1.2) -- (5,1.2);
\node at (0,0)[anchor=north east]{0};
\node at (0,5)[anchor=south]{$r_1$};
\node at (1,0)[anchor=north] {$\frac{2}{3}\epsilon_0(R)$};
\node at (2,0)[anchor=north west] {$\epsilon_0(R)$};
\node at (0,1)[anchor=east] {$\frac{2}{3}\epsilon_1(R)$};
\node at (0,2)[anchor=east] {$\epsilon_1(R)$};
\node at (5,0)[anchor=west] {$r_0$};
\node at (3,2)[anchor=north west] {$\tilde{f}(r)=R$};
\end{tikzpicture}
\end{minipage}
\end{figure}
\end{prop}
\begin{proof}
For $v\in V$, let $\tilde{\epsilon}_v:(0,1]\rightarrow \mathbb{R}_{>0}$ be a smooth function such that $\tilde{\epsilon}_v'(R)>0$ and
\begin{equation}
\frac{3}{4}\epsilon_v(R)<\tilde{\epsilon}_v(R)<\epsilon_v(R),
\end{equation}
and let $c_v=\tilde{\epsilon}_v(1)$.

Let $\tilde{g}:\mathbb{R}\rightarrow \mathbb{R}_{\geq 0}$ be a smooth function such that $\tilde{g}(x)=0$ for $x\leq 0$, $\tilde{g}'(x)>0$ for $x>0$, and $\tilde{g}(1)=1$ (see Figure \ref{tildeg}).

Let $\tilde{F}:(0,1]\times\mathbb{R}_{>0}^V\rightarrow \mathbb{R}$ be the smooth function given by
\begin{equation}
\tilde{F}(R,x)=R-1+\sum_{v\in V}\tilde{g}\left(9-\frac{9}{\tilde{\epsilon}_v(R)}x_v\right).
\end{equation}

If $\tilde{F}(R,x)=0$ and $R\in (0,1]$, $x_v\geq \frac{8}{9}\tilde{\epsilon}_v(R)\geq \frac{2}{3}\epsilon_v(R)$ for all $v\in V$, and $x_v\leq \tilde{\epsilon}_v(R)<\epsilon_v(R)$ for some $v\in V$.

Fix $x\in \mathbb{R}_{>0}^V$. As $R\to 0$, $\tilde{F}(R,x)\to -1$, and $\tilde{F}(1,x)\geq 0$, so for some $R\in (0,1]$, $\tilde{F}(R,x)=0$. Since $\frac{\partial{\tilde{F}}}{\partial R}>0$, there exists a smooth function $\tilde{f}:\mathbb{R}_{>0}^V\rightarrow (0,1]$ such that $\tilde{F}(\tilde{f}(x),x)=0$ (see Figure \ref{tildef}).
For $I\subset V$, let $\tilde{s}_I:\mathbb{R}_{>0}^I\rightarrow \mathbb{R}_{>0}^V$ be the map given by 
\begin{equation}
\tilde{s}_I(x)_v=\begin{cases}
x_v & v\in I
\\c_v & v\notin I
\end{cases}
\end{equation}
and let
\begin{equation}
\tilde{f}_I=\tilde{f}\circ \tilde{s}_I.
\end{equation}

The functions $\tilde{f}_I$ satisfy Properties 1-5 by construction. Properties 4 and 5 imply that for $I,J\subset V$, $\tilde{f}_J\circ r_J|_{U_{I\cup J}}-\tilde{f}_I\circ r_I|_{U_{I\cup J}}$ is supported in $\bigcup_{v\in I\cup J\setminus I\cap J}N_v$, which is disjoint from $\mathring{U}_I\cap \mathring{U}_J$, so we can indeed define a function $f:X\rightarrow (0,1]$ which satisfies Property 6.

\end{proof}

\subsection{Existence of top-level Hamiltonians}

\begin{set}
\label{top_level}
Let $V$ be a stratified symplectic subvariety of dimension $2(n-1)$ in a $2n$-dimensional symplectic manifold $(M,\omega)$, and fix $v\in V_{n-1}$.

Suppose we have radial Hamiltonians $r_u:U_u\rightarrow \mathbb{R}_{\geq 0}$ for $u<v$ (with respect to the subvariety $V$), which form a system of commuting Hamiltonians for $V_{<v}$.

Let $\nu=T\mathring{D}_v^\omega\subset T_M|_{\mathring{D}_v}$ denote the symplectic normal bundle of $\mathring{D}_v\subset M$.

Let $\pi:\nu\rightarrow \mathring{D}_v$ denote the bundle projection, and $s_0:\mathring{D}_v\rightarrow \nu$ the zero-section.

For $u<v$, the Hamiltonian $\mathbb{R}/\mathbb{Z}$-action generated by $r_u$ induces an $\mathbb{R}/\mathbb{Z}$-action on $\pi^{-1}(U_u\cap \mathring{D}_v)$ by symplectic bundle isomorphisms covering the Hamiltonian $\mathbb{R}/\mathbb{Z}$-action generated by $r_u|_{U_u\cap\mathring{D}_v}$.
\end{set}

\begin{lem}
\label{normal_bundle_symplectic}
Suppose we are in setting \ref{top_level}.

There exists a symplectic form $\omega^\nu$ on $\nu$ such that 
\begin{itemize}
\item $\omega^\nu_{s_0(x)}=\omega_x$ for $x\in \mathring{D}_v$.
\item $\omega^\nu|_{\pi^{-1}(x)}$ agrees with the standard symplectic form on $T_x\mathring{D}_v^\omega$.
\item With respect to any compatible metric on the fibres of $\nu$, the function $r(\xi)=\frac{1}{2}||\xi||^2$ generates the standard $U(1)$-action on $\nu$.
\item Any symplectic vector bundle isomorphism covering a symplectomorphism of $(\mathring{D}_v,\omega)$ is a symplectomorphism of $(\nu,\omega^\nu)$.
\end{itemize}
\end{lem}

\begin{proof}
We can define a  $Sp(2)$-invariant one-form $\sigma\in\Omega^1(\nu)$ by letting $\sigma_{\xi}=\iota_{\xi}\omega_{x}$ for $\xi\in T_x\mathring{D}_v^\omega\subset \nu$, under the natural identification $T_x^*M\cong T_\xi^*\nu$.

We can define a closed two-form
\begin{equation}
\omega^\nu=\pi^*\omega+\frac{1}{2}d\sigma.
\end{equation}

Fix a compatible unitary structure on $\nu$. 

For any $\xi\in T_x\mathring{D}_v^\omega$, $\frac{1}{2}(d\sigma)_\xi(\xi \wedge i\xi)=||\xi||^2$, so $\frac{1}{2}d\sigma$ is equal to the standard symplectic form on the fibre. It follows that $\omega_\nu$ is a symplectic form.

The fibre-wise $U(1)$-action is the flow of the vector field $i\xi$, and $\iota_{i\xi}\omega^\nu_\xi=-(dr)_\xi$, so $r$ indeed generates the $U(1)$-action.
\end{proof}
\begin{lem}[Invariant Tubular Neighbourhood]
\label{tubular_nbhd}
Suppose we are in Setting \ref{top_level}.

Then there exists a compatible unitary structure for $\nu$, a tubular neighbourhood $T\subset \nu$ of the zero-section, an embedding 
\begin{equation}
\iota:T\rightarrow M
\end{equation}
and for $u<v$, $r_u$-invariant intersection neighbourhoods $I_u\subset \iota(T)\cap U_u$ for the pair $u<v$, such that 
\begin{itemize}
\item The $\mathbb{R}/\mathbb{Z}$-action on $\pi^{-1}(I_u\cap\mathring{D}_v)$ induced by the Hamiltonian $\mathbb{R}/\mathbb{Z}$-action generated by $r_u$ on $I_u$ is an action by unitary line bundle isomorphisms for $u<v$,
\item $\iota\circ s_0=id_{\mathring{D}_v}$,
\item $d\iota_{s_0(x)}=id_{T_xM}$ for $x\in \mathring{D}_v$, with respect to the natural identification of the tangent spaces
\item $\iota^{-1}(I_u)\subset \pi^{-1}(I_u\cap\mathring{D}_v)$ for $u<v$,
\item $\iota|_{\iota^{-1}(I_u)}$ is equivariant with respect to the Hamiltonian $\mathbb{R}/\mathbb{Z}$-action generated by $r_u$ on $I_u$ and the induced action on $\pi^{-1}(I_u\cap\mathring{D}_v)$ for $u<v$.
\end{itemize}
\end{lem}

\begin{proof}
Let $X_v=M\setminus |V_{<v}|$.

By Lemma \ref{average_construction}, we obtain an almost-complex structure $J$ and metric $g$ on $X_v$ such that $(J,\omega,g)$ is a compatible triple and $\mathring{D}_v$ is an almost-complex submanifold with respect to $J$, along with $r_u$-invariant stratum neighbourhoods $T_u\subset U_u$ such that $g|_{T_u\cap X_v}$ is $r_u$-invariant, for $u<v$.

The metric $g$ induces a unitary structure on the fibres of $\nu$, so for each $u<v$, the $\mathbb{R}/\mathbb{Z}$-action on $\pi^{-1}(T_u\cap\mathring{D}_v)$ induced by $\phi_{r_u}^t$ is an action by unitary line bundle isomorphisms.

Let $T\subset \nu$ be a tubular neighbourhood of $s_0(\mathring{D}_v)$ such that exponential map $\exp:T\rightarrow M$ obtained from the metric $g$ is an embedding.  By definition of $\exp$, $d\exp_{s_0(x)}=id_{T_xM}$, with respect to the natural identification of $T_{s_0(x)}\nu$ with $T_x\mathring{D}_v\oplus T_x\mathring{D}_v^\omega=T_xM$.

Because for each $u<v$, $g|_{T_u}$ is invariant under $\mathbb{R}/\mathbb{Z}$-action $\phi_{r_u}^t$, which also preserves $\mathring{D}_v\cap T_u$ , there exists some $r_u$-invariant neighbourhood $I_u\subset \exp(\pi^{-1}(\mathring{D}_v\cap T_u)\cap T)$ of $\mathring{D}_v\cap T_u$ such that $\exp|_{\exp^{-1}(I_u)}:\exp^{-1}(I_u)\rightarrow I_u$ is equivariant with respect to $\phi_{r_u}^t$ and the induced action. Note that $\exp(I_u)$ is an intersection neighbourhood for the pair $u<v$.
\end{proof}

\begin{ex}
\label{cubic_tubular_nbhd}
Consider the cubic curve from Example \ref{cubic_Example}. We could apply Lemma \ref{tubular_nbhd} to the vertex $1$ corresponding to the smooth locus of the curve, and the radial Hamiltonian $r_0$ constructed in \ref{cubic_circle_action}.

Recall that the smooth locus is given by an embedding $z\mapsto [z^3:1:z]$ of $\mathbb{C}$. The circle action generated  by $r_0$ induces a circle action on $\mathbb{C}$ given by $t\cdot z=e^{-2\pi i t}z$. This action extends to an induced action on the symplectic normal bundle $\nu$. The Lemma says that this action is, in fact, unitary with respect to some choice of compatible almost-complex structure. 

In the proof of the Lemma, we choose a compatible almost-complex structure, then obtain an invariant one by averaging. If we choose to start with the standard complex structure, which is, in this case, preserved by the circle action, our choice of invariant almost-complex structure is the standard one.

This means that our choice of invariant unitary structure on $\nu$ comes from the complex structure on $\mathbb{CP}^2$.

The Lemma says that we can choose a tubular neighbourhood $T$ in $\nu$ such that the exponential map $\iota:T\rightarrow \mathbb{CP}^2$ with respect to our invariant compatible metric is an embedding, and there exists an intersection neighbourhood $I$, like the one pictured in figure \ref{cubic_intersection_neighbourhood}, such that the restriction $\iota|_{\iota^{-1}(I)}$ is equivariant with respect to the circle action induced by $r_0$.

In this case, $\iota$ is just the exponential map obtained from the Fubini-Study metric, which is, in fact, globally equivariant.
\end{ex}

\begin{lem}
\label{invariant_Moser}
Suppose we are in Setting \ref{top_level}.

Let $\omega^\nu$ be a symplectic form on $\nu$ as in the statement of Lemma \ref{normal_bundle_symplectic}

Let $\iota:T\rightarrow M$ satisfy all the properties in the statement of Lemma \ref{tubular_nbhd}. Then there exists a tubular neighbourhood $S\subset T$ of the zero-section and a map 
\begin{equation}
\phi:\iota(S)\rightarrow \iota(T)
\end{equation}
such that 
\begin{itemize}
\item $\phi|_{\mathring{D}_v}=id_{\mathring{D}_v}$,
\item $(\phi\circ \iota|_S)^*\omega=\omega^\nu$,
\item for each $u<v$, there exists an $r_u$-invariant intersection neighbourhood $J_u\subset \iota(S)\cap U_u$ such that $(\phi\circ \iota|_S)^{-1}|_{J_u}$ is equivariant with respect to the Hamiltonian $\mathbb{R}/\mathbb{Z}$-action generated by $r_u$ on $U_u$ and the induced action on $\pi^{-1}(U_u\cap\mathring{D}_v)$.
\end{itemize}
\end{lem}

\begin{proof}
Because $T$ deformation retracts onto the zero-section $s_0(\mathring{D}_v)$, and $(\iota_*\omega_\nu -\omega)_x=0$ for $x\in \mathring{D}_v$, we have that $d\mu=\iota_*\omega^\nu-\omega$ for some $\mu\in\Omega^1(\iota(T))$ such that $\mu_x=0$ for $x\in\mathring{D}_v$. 

Let $I_u$ for $u<v$ be intersection neighbourhoods as in the statement of Lemma \ref{tubular_nbhd}.

Define a function $r:\nu\rightarrow \mathbb{R}_{\geq 0}$ by $r(x)=\frac{1}{2}||x||_g^2$, where $g$ is a metric on $\nu$ as in the statement of Lemma \ref{tubular_nbhd}. For $u<v$, $\{r\circ \iota^{-1}|_{I_u},r_u|_{I_u}\}=0$, so by Lemma \ref{invariant_neighbourhoods} applied to $r\circ \iota^{-1}|_{\iota(T)}$ together with $r_u$ for $u<v$, we can choose a stratum neighbourhood $Y\subset \iota(T)$ for $v$ and stratum neighbourhoods $N_u\subset U_u$ for $u<v$ such that $Y\cap N_u$ is $r_u$-invariant. 

By Lemma \ref{average_construction}, we may assume that, for some $r_u$-invariant stratum neighbourhoods $T_u\subset N_u$ for $u<v$, $\mu|_{T_u\cap Y}$ is $r_u$-invariant.

Let $\phi_\mu^t$ denote the time $t$ flow of the vector field $X_\mu$ on $Y$ $\omega$-dual to $\mu$, where defined. We have that $\phi^t_\mu|_{\mathring{D}_v}=id_{\mathring{D}_v}$, and $(\phi_\mu^1)^*\iota_*\omega^\nu=\omega$ where defined.

Choose a tubular neighbourhood $S\subset \iota^{-1}(Y)$ of $s_0(\mathring{D}_v)$ such that $\phi_\mu^t(\iota(S))\subset U$ for all $t\in [0,1]$. We have that 
\begin{equation}
\phi_\mu^1\circ \iota|_S:S\rightarrow \phi_\mu^1(\iota(S))
\end{equation}
is a symplectomorphism.

Because for $u<v$, $X_\mu|_{W_u\cap U}$ is $r_u$-invariant, and the set $\mathring{D}_v\cap W_u$ is fixed by the flow of $X_\mu$ and $r_u$-invariant, there exists some $r_u$-invariant neighbourhood $J_u\subset \phi^1_\mu(\iota(S))\cap W_u$ of $\mathring{D}_v\cap W_u$ such that $(\phi_\mu^1)^{-1}|_{J_u}$ is equivariant with respect to $\phi_{r_u}^t$. 

Note that $J_u$ is an intersection neighbourhood for the pair $u<v$.

Because $\iota^{-1}|_{W_u\cap \iota(S)}$ is also equivariant with respect to the action generated by $r_u$ on $W_u$ and the induced action on $\pi^{-1}(W_u\cap\mathring{D}_v)$, letting $\phi=\phi_\mu^1$, we have that $\phi\circ \iota^{-1}|_{J_u}$ is equivariant as required.
\end{proof}

\begin{ex}
\label{cubic_Moser}
Returning to the cubic curve in Example \ref{cubic_Example}, Lemma \ref{invariant_Moser} says that we can deform the tubular neighbourhood $\iota:T\rightarrow \mathbb{CP}^2$ of the smooth locus described in Example \ref{cubic_tubular_nbhd}, so that it becomes a symplectomorphism, and remains equivariant with respect to the circle action generated by $r_0$ near the singular point.

Because we only need to consider one circle action, which extends globally, Lemma \ref{invariant_Moser} reduces in this case to an equivariant version of the Moser trick.
\end{ex}

\begin{prop}
\label{hypersurf_ham}
Given a stratified symplectic subvariety of dimension $2(n-1)$ in a $2n$-dimensional symplectic manifold $(M,\omega)$, let $r_v:U_v\rightarrow \mathbb{R}_{\geq 0}$ for $v\in V_{\leq n-2}$ be radial Hamiltonians (with respect to the subvariety $V$, not just $V_{\leq n-2}$), which form a system of commuting Hamiltonians for $V_{\leq n-2}$.

Suppose for each $v\in V_{n-1}$, $\mathring{D}_v$ is an almost-complex submanifold with respect to some $\omega$-compatible almost-complex structure on $M$.

Then we may choose radial Hamiltonians $r_v:U_v\rightarrow \mathbb{R}_{\geq 0}$ for $v\in V_{n-1}$ such that $r_v$ for $v\in V$ is a system of commuting Hamiltonians for $V$.
\end{prop}

\begin{proof}
Fix $v\in V_{n-1}$. We will construct a radial Hamiltonian for $v$ and show that it commutes with $r_u$ for $u<v$.

Let $\iota:T\rightarrow M$ and $\phi:\iota(S)\rightarrow \iota(T)$ be as in the statements of Lemmas \ref{tubular_nbhd} and \ref{invariant_Moser}.

Let $g$ be a metric on $\nu$ as in Lemma \ref{tubular_nbhd}.

The Hamiltonian $r(x)=\frac{1}{2}||x||_g^2$ generates a Hamiltonian $\mathbb{R}/\mathbb{Z}$-action on the fibres of $\nu$. Choose a neighbourhood $U_v\subset \phi(\iota(T))$ of $\mathring{D}_v$ such that $(\phi\circ\iota)^{-1}(U_v)$ is preserved by this action.

Let 
\begin{equation}
r_v=r\circ \iota^{-1}\circ \phi^{-1}:U_v\rightarrow \mathbb{R}_{\geq 0}
\end{equation}

Because $\phi\circ \iota$ is a symplectomorphism, $r_v$ generates an effective Hamiltonian $\mathbb{R}/\mathbb{Z}$-action on $U_v$, and this action fixes exactly $r_v^{-1}(0)=\mathring{D}_v$, so $r_v$ is a radial Hamiltonian for $v$.

Since for each $u<v$, we have an intersection neighbourhood $J_u\subset U_u\cap U_v$ such that $(\phi\circ \iota)^{-1}|_{J_u}$ is equivariant with respect to the Hamiltonian $\mathbb{R}/\mathbb{Z}$-action on $U_u$ and the induced action on $\pi^{-1}(U_u\cap \mathring{D}_v)$, and the induced action is unitary (i.e. preserves $r$), we have that $\{r_v|_{J_u\cap U_v},r_u|_{J_u\cap U_v}\}=0$, so $r_v$ commutes with $r_u$ as required.
\end{proof}

\begin{ex}
\label{cubic_sch}
In the case of the cubic curve from Example \ref{cubic_Example}, Proposition \ref{hypersurf_ham} allows us to construct a radial Hamiltonian $r_1$ for the vertex $1$ corresponding to the smooth locus of the curve, which commutes with $r_0$ (see Example \ref{cubic_circle_action}).

As discussed in Example \ref{cubic_Moser}, the preceding Lemmas allow us to construct a symplectic tubular neighbourhood of the smooth locus, which is equivariant with respect to the circle action generated by $r_0$. Furthermore, the standard complex structure on the normal bundle $\nu$ is invariant under the induced circle action, so there is a natural $U(1)$ action on the fibres of $\nu$, which commutes with the action generated by $r_0$. This action is Hamiltonian, generated by the norm on the fibres of $\nu$, which gives us our radial Hamiltonian $r_1$.
\end{ex}
\section{Adapted Primitives}
\label{s4}
\subsection{Adapted Primitives}
\begin{set}
\label{prim}
Let $V$ be a stratified symplectic subvariety in $(M,\omega)$ and $X=M\setminus |V|$, and $r_v:U_v\rightarrow \mathbb{R}_{\geq 0}$ for $v\in V$ a system of commuting Hamiltonians for $V$. 

Let $\theta\in\Omega^1(X)$ be a primitive for $\omega$, i.e.
\begin{equation}
d\theta=\omega|_X.
\end{equation}

Let $Z$ be the Liouville vector field on $X$ associated to $\theta$, i.e.
\begin{equation}
\iota_Z\omega=-\theta.
\end{equation}
\end{set}

\begin{lem}
\label{action_constant}
Let $\theta$ be a primitive as in Setting \ref{prim}.

For $v\in V$, let
\begin{equation}
\kappa_v(x)=r_v(x)-\int_{t=0}^1\left(\iota_{X_{r_v}}\theta\right)_{\phi_{r_v}^t(x)}dt.
\end{equation}
Then $\kappa_v$ is constant on $U_v\cap X$.
\end{lem}

\begin{proof}
\begin{equation}
\begin{split}
d\kappa_v&=dr_v-\int_{t=0}^1(\phi_{r_v}^t)^*d(\iota_{X_{r_v}}\theta)dt
\\&=dr_v-\int_{t=0}^1(\phi_{r_v}^t)^*(\mathcal{L}_{X_{r_v}}\theta-\iota_{X_{r_v}}\omega)dt
\\&=dr_v-\int_{t=0}^1\left(\frac{d}{dt}\left((\phi_{r_v}^t)^*\theta\right)+dr_v\right)dt
\\&=0.
\end{split}
\end{equation}

\end{proof}

\begin{defn}[Action of a Radial Hamiltonian]
Let $\theta$ be a primitive as in Setting \ref{prim}.

For $v\in V$, the action of $r_v$ with respect to $\theta$, is the real number 
\begin{equation}
\kappa_v=r_v(\gamma)+\int_{\mathbb{R}/\mathbb{Z}}\gamma^*\theta
\end{equation}
where $\gamma(t)=\phi_{r_v}^t(x)$ for some $x\in U_v\cap X$ is any 1-periodic Hamiltonian orbit generated by $r_v|_{U_v\cap X}$.

Note that this is well-defined by Lemma \ref{action_constant}.
\end{defn}

\begin{lem}
\label{rel_deRham}
Let $\theta$ be a primitive as in Setting \ref{prim}. 

For $v\in V_k$, $\mathring{D}_v$ represents a homology class in $H_{2k}(|V|)$. Let $PD^{rel}(\mathring{D}_v)\in H^{2(n-k)}(M,X)$ denote the cohomology class represented by $\mathring{D}_v$ via Poincaré duality.

Then 
\begin{equation}
[\omega,\theta]=\sum_{v\in V_{n-1}} \kappa_v PD^{rel}(\mathring{D}_v)\in H^2(M,X).
\end{equation}
\end{lem}

\begin{proof}
The Poincaré duality map gives an isomorphism
\begin{equation}
H_{2n-2}\left(|V|\right)\rightarrow H^2\left(M,X\right),
\end{equation}
and the homology classes represented by the submanifolds $\mathring{D}_v$ for $v\in V_{n-1}$ generate the left hand side, so the classes $PD^{rel}(\mathring{D}_v)$ for $v\in V_{n-1}$ generate the right hand side. 

For any $v\in V_{n-1}$, by considering a union of $\mathbb{R}/\mathbb{Z}$-orbits of $\phi_{r_v}^t$ converging to a point in $\mathring{D}_v$, we can construct a disc $u\in \pi_2(X\cup \mathring{D}_v,X)$ such that $[u]\cdot [\mathring{D}_v]=1$, and
\begin{equation}
\int_u\omega-\int_{\partial u}\theta=\kappa_v.
\end{equation} 
This shows that the coefficient of $PD^{rel}(\mathring{D}_v)$ in $[\omega,\theta]$ is indeed $\kappa_v$.
\end{proof}

\begin{defn}[Adapted One-form]
Let $\theta$ be a primitive as in Setting \ref{prim}.

We say that $\theta$ is adapted at $v\in V$ if, on some stratum neighbourhood of $v$,
\begin{equation}
\iota_Zdr_v=r_v-\kappa_v.
\end{equation} 
where $\kappa_v$ is the action of $r_v$.

We say that $\theta$ is adapted if for each $v\in V$, $\theta$ is adapted at $v$.
\end{defn}

\begin{defn}[Weakly Adapted One-form]
Let $\theta$ be a primitive as in Setting \ref{prim}.

We say that $\theta$ is weakly adapted at $v\in V$ if, on some stratum neighbourhood of $v$,
\begin{equation}
\iota_Zdr_v<0.
\end{equation} 

We say that $\theta$ is weakly adapted if for each $v\in V$, $\theta$ is weakly adapted at $v$.

Let $\kappa_v$ denote the action of $r_v$ with respect to $\theta$. Note that if $\kappa_v>0$ and $\theta$ is adapted at $v$, then $\theta$ is weakly adapted at $v$. Conversely, if $\theta$ is weakly adapted and adapted at $v$, then $\kappa_v>0$ by definition of the action.
\end{defn}

\begin{ex}
The cubic curve in Example \ref{cubic_Example} is anticanonical, so the compactification one-form 
\begin{equation}
\theta=\kappa d^c \log\left(\frac{|z_0z_1^2-z_2^3|^2}{(|z_0|^2+|z_1|^2+|z_2|^2)^3}\right)
\end{equation}
is a primitive for the symplectic form $\omega$, where $[\omega]=2\kappa c_1(T\mathbb{CP}^2)$.

Recall the radial Hamiltonian $r_0$ from Example \ref{cubic_circle_action}, which by Proposition \ref{hypersurf_ham}, forms part of a system of commuting Hamiltonians for $\{0,1\}$. $\theta$ is invariant under the circle action generated by $r_0$, so by Equation \eqref{adapted_equivariant} $\theta$ is adapted at $0$.

An orbit of the circle action bounds a disc centred at $[1:0:0]$ inside the holomorphic plane $z\mapsto [-1:z^3:z^2]$, so the action of $r_0$ is $6\kappa$, the intersection number of the disc with the cubic curve ($H_2(\mathbb{CP}^2,X;\mathbb{R})$ is generated by the Poincaré dual of the class of the curve).
\end{ex}

\subsection{Convexity}

\begin{prop}
Let $\theta$ be a weakly adapted primitive for $\omega$.

Then $(X,\theta)$ has the structure of a finite type convex symplectic manifold, in the sense of Section 8 of \cite{McL_growth}.
\end{prop}

\begin{proof}
\label{convexity}
We have that for each $v\in V$, $\iota_Zdr_v<0$ on some stratum neighbourhood $T_v\subset U_v$ for $v$. By Proposition \ref{exhaustion_existence} applied to the functions $r_v|_{T_v}$, there exists a smooth, proper function 
\begin{equation}
f:X\rightarrow (0,1],
\end{equation}
and by properties 2,3 and 7, whenever $f(x)<1$, $\iota_{Z_x}(df)_x<0$.
\end{proof}

\begin{prop}
\label{cohomologous_adapted}
Let $\theta_0,\theta_1$ be primitives as in Setting \ref{prim}.

Suppose that $\theta_i$ is weakly adapted for $i=0,1$ and $\theta_1-\theta_0$ is exact (by Lemma \ref{rel_deRham}, this is equivalent to the condition that for $v\in V_{n-1}$, the actions of $r_v$ with respect to $\theta_0$ and $\theta_1$ are equal).

Then there is a strong deformation equivalence between $(X,\theta_0)$ and $(X,\theta_1)$, in the sense of Section 8 of \cite{McL_growth}.
\end{prop}

\begin{proof}
Let $Z_0$ and $Z_1$ be the associated Liouville vector fields on $X$ to $\theta_0$ and $\theta_1$ respectively.

For $t\in [0,1]$, let $\theta_t=(1-t)\theta_0+t\theta_1$, so $\theta_t-\theta_0$ is exact. The Liouville vector field associated to $\theta_t$ is $Z_t=(1-t)Z_0+tZ_1$.

We have that for each $v\in V$, $\iota_{Z_i}dr_v<0$ for $i=0,1$ on some stratum neighbourhood $T_v\subset U_v$ for $v$.

As in the proof of Lemma \ref{convexity}, by applying \ref{exhaustion_existence} to the functions $r_v|_{T_v}$ for $v\in V$ we obtain an exhaustion function $f:X\rightarrow (0,1]$ such that whenever $f<1$, $\iota_{Z_i}df<0$ for $i=0,1$, so $df(Z_t)<0$ for $t\in [0,1]$, so $\theta_t$ is a finite type strong convex deformation equivalence.
\end{proof}

\subsection{Existence}

\begin{cor}
\label{adapted_existence}
Let $\theta$ be a primitive as in Setting \ref{prim}.

Suppose $\theta$ is weakly adapted.

Then there exists an adapted primitive $\overline{\theta}$ for $\omega$, and a finite type deformation equivalence between $(X,\theta)$ and $(X,\overline{\theta})$.

Note that this implies that for each $v\in V$, the action of $r_v$ with respect to $\theta$ and $\overline{\theta}$ is equal.
\end{cor}

\begin{proof}
By Lemma \ref{average_construction}, there exists a one-form $\overline{\theta}\in\Omega^1(X)$, such that $\overline{\theta}-\theta$ is exact, and for each $v\in V$, an $r_v$-invariant stratum neighbourhood $T_v\subset U_v$, such that $\overline{\theta}|_{W_v\cap X}$ is $r_v|_{T_v}$-invariant.

Let $\overline{Z}$ be the Liouville vector field associated to $\overline{\theta}$.

For each $v\in V$, 
\begin{equation}
\label{adapted_equivariant}
d(\iota_{\overline{Z}}
dr_v-r_v)=d(\iota_{X_{r_v}}\overline{\theta}-r_v)=-\iota_{X_{r_v}}\omega+\mathcal{L}_{X_{r_v}}\overline{\theta}-dr_v=\mathcal{L}_{X_{r_v}}\overline{\theta}=0
\end{equation}
in the region $T_v\cap X$, so $\iota_{\overline{Z}}dr_v=r_v-\kappa_v$ on $T_v$ for some constant $\kappa_v$, which is in fact the action of $r_v$, with respect to both $\theta$ and $\overline{\theta}$ (so $\kappa_v>0$).

By Lemma \ref{cohomologous_adapted}, $\theta$ is strongly convex deformation equivalent to $\overline{\theta}$. 
\end{proof}
\section{Construction of Hamiltonian}
\label{s5}
\subsection{Construction of $\rho^R$}

\begin{lem}[Sufficiently Small Stratum Neighbourhoods]
\label{small_nbhds}
Let $V$ be a stratified symplectic subvariety in $(M,\omega)$. 

Let $r_v:U_v\rightarrow \mathbb{R}_{\geq 0}$ for $v\in V$ be a system of commuting Hamiltonians for $V$, with weights $w_v$ for $v\in V$.

Let $X=M\setminus |V|$. Let $\theta\in\Omega^1(X)$ be an adapted primitive for $\omega$, $Z$ the Liouville vector field $\omega$-dual to $\theta$, $\psi^t$ the time-$t$ flow of $Z$, and $L$ the Lagrangian skeleton of $(X,\theta)$.

Let $\kappa_v>0$ denote the action of $r_v$ with respect to $\theta$.

Fix once and for all a (not necessarily $\omega$-compatible) metric on $M$, and let $inj(M)$ denote the radius of injectivity.

Then we may assume that the neighbourhoods $U_v$ satisfy
\begin{enumerate}
\item $U_I$ is $r_I$-invariant for all $I\subset V$, 
\item $\iota_Zdr_v=r_v-\kappa_v$ on $U_v\cap X$
\item for all $v\in V$ and $x\in U_v$, $d(x,\mathring{D}_v)<\frac{1}{4}inj(M)$,
\item for any orbit $\gamma\subset U_v$ of the Hamiltonian $\mathbb{R}/\mathbb{Z}$-action generated by $r_v$, $diam(\gamma)<\frac{1}{4n}inj(M)$,
\item for all $t>0$, if $\psi^{-t}(x)\in U_v\cap X$, then $x\in U_v$.
\item for each $v\in V$, $U_v$ is an open subset of $M\setminus L$.
\item for $u,v\in V$, either $u\leq v$, $v\leq u$, or $U_u\cap U_v=\emptyset$.
\end{enumerate}
\end{lem}

\begin{proof}
Conditions 2-4 and 7 are maintained when passing to smaller stratum neighbourhoods. We can therefore, after passing to smaller neighbourhoods, assume that conditions 2-4 and 7 are satisfied (this follows from adaptedness of $\theta$ for 2, continuity for 3, the fact that as $x\to y\in\mathring{D}_v$, $\phi^t_{r_v}(x)\to y$ for 4, and Lemma \ref{intersection_commute} for 7), and also that $U_v\subset M\setminus L$. 

Applying Corollary \ref{invariant_neighbourhoods}, we obtain stratum neighbourhoods which additionally satisfy Condition 1.

We can now replace $U_v$ by
\begin{equation}
U_v\setminus \bigcup_{t\geq 0}\psi^{-t}(X\setminus U_v).
\end{equation}

This preserves conditions 1-4 and 7 and also satisfies 5.

Lastly, we can replace $U_v$ by its interior, which preserves conditions 1-5 and 7 and also satisfies 6.
\end{proof}

\begin{set}
\label{main_setting}
Let $V$ be a stratified symplectic subvariety in $(M,\omega)$. 

Let $r_v:U_v\rightarrow \mathbb{R}_{\geq 0}$ for $v\in V$ be a system of commuting Hamiltonians for $V$, with weights $w_v$ for $v\in V$.

Let $X=M\setminus |V|$. Let $\theta\in\Omega^1(X)$ be an adapted primitive for $\omega$, $Z$ the Liouville vector field $\omega$-dual to $\theta$, $\psi^t$ the time-$t$ flow of $Z$, and $L$ the Lagrangian skeleton of $(X,\theta)$.

Let $\kappa_v>0$ denote the action of $r_v$ with respect to $\theta$.

We assume that the stratum neighbourhoods $U_v$ satisfy Conditions 1-5 from Lemma \ref{small_nbhds}.

Let $\epsilon_v:[0,1]\rightarrow \mathbb{R}_{\geq 0}$ for $v\in V$ be continuous functions as in Lemma \ref{compact_boxes}, such that whenever $u<v$, $\epsilon_v<\frac{\kappa_v}{3\kappa_u}\epsilon_u$. Choose $R_0\in (0,1)$ such that $\epsilon_v(R_0)<\kappa_v$ for all $v\in V$.

Let $f:X\rightarrow (0,1]$, be an exhaustion function, and $\tilde{f}_I:\mathbb{R}_{>0}^I\rightarrow (0,1]$ be smooth functions as in the statement of Proposition \ref{exhaustion_existence}.

Now, for $R\in (0,R_0)$, let
\begin{equation}
Y^R=f^{-1}(R),
\end{equation}

For $I\subset V$, let
\begin{equation}
\tilde{B}_I=\prod_{v\in I}[0,\kappa_v)
\end{equation}
and let
\begin{equation}
\tilde{Y}_I^R=\tilde{f}_I^{-1}(R)\cap \tilde{B}_I.
\end{equation}

Let $\tilde{Z}_I$ denote the vector field on $\mathbb{R}^I$ given by
\begin{equation}
(\tilde{Z}_I)_x=\sum_{v\in I}(x_v-\kappa_v)\partial_v.
\end{equation}

By Properties 2 and 3 from Proposition \ref{exhaustion_existence}, $\iota_{\tilde{Z}_I}d\tilde{f}_I>0$ in $\tilde{B}_I$, so we can define a smooth function 
\begin{equation}
P_I^R:\tilde{B}_I\rightarrow \tilde{Y}_I^R
\end{equation}
such that $P_I^R(x)$ is the unique point of intersection of $\tilde{Y}_I^R$ and the flowline of $\tilde{Z}_I$ containing $x$ (i.e. $P_I^R$ is projection onto $\tilde{Y}_I^R$ via the flow of $\tilde{Z}_I^R$).

Let 
\begin{equation}
U_I^{\max}=U_I\cup\bigcup_{t>0}\psi^{-t}(U_I\cap X)\subset M\setminus L
\end{equation}
and let 
\begin{equation}
r_I^{\max}:U_I^{\max}\rightarrow \tilde{B}_I
\end{equation}
be the unique smooth function such that $r_I^{\max}|_{U_I}=r_I$, and $\iota_zdr_I^{\max}=\tilde{Z}_I$.
\end{set}

\begin{lem}
\label{open_cover}
Suppose we are in Setting \ref{main_setting}.

Then there exist open sets $\mathring{U}_I^{R,\max}\subset U_I^{\max}$ for $I\subset V$, which cover $M\setminus L$, such that 
\begin{equation}
\label{Y_charts}
Y^R\cap \mathring{U}_I^{R,\max}=(r_I^{\max})^{-1}(\tilde{Y}_I^R) \cap \mathring{U}_I^{R,\max}
\end{equation}
and 
\begin{equation}
\label{open_cover_intersection}
\mathring{U}_I^{R,\max}\cap \mathring{U}_J^{R,\max}\subset\left(P_{I\cup J}^R\circ r_{I\cup J}^{\max}\right)^{-1}\left(\tilde{Y}_{I\cup J}^R\cap \prod_{v\in I\cup J\setminus I\cap J}[\epsilon_v(R),\infty) \times\mathbb{R}^{I\cap J}\right)
\end{equation}
for $I,J\subset V$.
\end{lem}

\begin{proof}
For $v\in V$, let $N_v^R\subset U_v$ be as in Equation \eqref{N_v^R}. We can now define a closed subset $N_v^{R,\max}\subset U_v^{\max}$ by letting
\begin{equation}
N_v^{R,\max}=N_v^R\cup\bigcup_{t>0}\psi^{-t}(N_v^R\cap X).
\end{equation}

Let
\begin{equation}
\mathring{U}_I^{R,\max}=U_I^{\max}\cap \bigcap_{v\in V\setminus I} \left(M\setminus N_v^{R,\max}\right).
\end{equation}

The fact that the sets $\mathring{U}_I^{R,\max}$ cover $M\setminus L$ follows from the fact that the sets $N_v^R$ cover $|V|$, so the sets $N_v^{R,\max}$ cover $M\setminus L$.

Equation \eqref{Y_charts} follows from Property 5 of the functions $\tilde{f}_I$.

Suppose $I\subset J\subset V$, and $x\in \mathring{U}_I^{R,\max}\cap \mathring{U}_J^{R,\max}$. Let $r=r_J^{\max}(x)$ and $r'=P_J^R(r)$. Let $u\in V$ denote the minimal element such that $x\in N_u^{R,\max}$. Possibly after flowing forward along $Z$, we may assume that $x\in U_J$, $r_u<\frac{1}{3}\epsilon_u(R)$, and $r_t'>r_t>\epsilon_v(R)$ for all $t<u$, so $u\in I$. By Condition 7 on the stratum neighbourhoods, either $u<v$ or $u>v$.

If $u<v$, $r'_v>\frac{\kappa_v}{3\kappa_u}\epsilon_u(R)$ (see Figure \ref{trapezium}), so $r_v'>\epsilon_v(R)$ by the condition on the functions $\epsilon_u,\epsilon_v$. If $u>v$, $r_v'>r_v>\epsilon_v(R)$ by assumption.
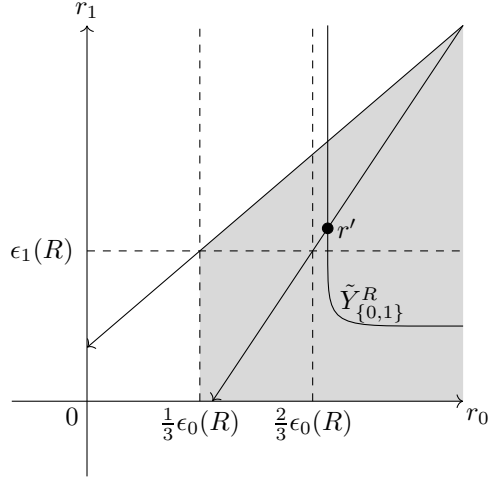
\begin{figure}
\caption{If $V=\{0<1\}$ and $r'_1<\epsilon_1(R)$, then the flowline of $\tilde{Z}_{\{0,1\}}$ containing $r$ ends in  $\left[\frac{1}{3}\epsilon_0(R),\infty\right)\times [0,\epsilon_1(R)]$.}
\label{trapezium}
\centering
\begin{tikzpicture}
\filldraw[fill=gray!30, draw=none](1.5,0)--(1.5,2)--(5,5)--(5,0)-- cycle;
\draw[->] (-1,0) -- (5,0);
\draw[->] (0,-1) -- (0,5);
\draw[dashed] (0,2)--(5,2);
\draw[dashed] (1.5,0)--(1.5,5);
\draw[dashed] (3,0)--(3,5);

\draw[->] (5,5)--(0,0.7143);
\draw[->] (5,5)--(1.666667,0);
\filldraw[black] (3.2,2.3) circle (2pt) node[anchor=west]{$r'$};

\draw (3.2,5) --(3.2,2) .. controls (3.2,1) .. (4.5,1) -- (5,1);
\node at (-0.2,-0.2) {0};
\node at (0,5.2) {$r_1$};
\node at (1.5,-0.3) {$\frac{1}{3}\epsilon_0(R)$};
\node at (3,-0.3) {$\frac{2}{3}\epsilon_0(R)$};

\node at (-0.6,2) {$\epsilon_1(R)$};
\node at (5.2,-0.2) {$r_0$};
\node at (3.8,1.3) {$\tilde{Y}_{\{0,1\}}^R$};
\end{tikzpicture}
\end{figure}

We have shown that 
\begin{equation}
\mathring{U}_I^{R,\max}\cap \mathring{U}_J^{R,\max}\subset (P_J^R\circ r_J^{\max})^{-1}\left(\tilde{Y}_J^R\cap \prod_{v\in J\setminus I}(\epsilon_v(R),\infty)\times\mathbb{R}^I\right),
\end{equation}
which, together with the fact that for any $I,J\subset V$, 
\begin{equation}
\mathring{U}_I^{R,\max}\cap \mathring{U}_J^{R,\max}\subset \mathring{U}_{I\cup J}^{R,\max}\cap \mathring{U}_{I\cap J}^{R,\max},
\end{equation}
implies Equation \eqref{open_cover_intersection}.
\end{proof}

\begin{prop}
\label{rho}
Suppose we are in Setting \ref{main_setting}, and let the open sets $\mathring{U}_I^{R,\max}$ be as in Lemma \ref{open_cover}.

Then there exists a smooth function 
\begin{equation}
\rho^R:M\setminus L\rightarrow \mathbb{R}
\end{equation}
and smooth functions
\begin{equation}
\nu^R_v:M\setminus L\rightarrow \mathbb{R}
\end{equation}
for $v\in V$, such that

\begin{enumerate}
\item For $I\subset V$, $\rho^R|_{U_I}$ and $\nu_v^R|_{U_I}$ for $v\in I$ are $r_I$-invariant,
\item $\rho^R|_{Y^R}=1$,
\item $\iota_Zd\rho^R=\rho^R$,
\item $\nu_v^R\geq 0$ for $v\in V$,
\item $\iota_Zd\nu_v^R=0$ for $v\in V$,
\item $d\rho^R=-\sum_{v\in V}\nu_v^Rdr_v^{\max}$.
\end{enumerate}
\end{prop}
\begin{proof}
For $I\subset V$, we can define smooth functions $\tilde{\rho}^R_I:\tilde{B}_I\rightarrow \mathbb{R}$ by letting 
\begin{equation}
\tilde{\rho}_I^R(r)=\frac{r_v-\kappa_v}{(P_I^R(r))_v-\kappa_v}
\end{equation}
for any $v\in I$. This automatically satisfies $\tilde{\rho}_I^R=1$ and $\iota_{\tilde{Z}_I}d\tilde{\rho}_I^R=\tilde{\rho}_I^R$.

Now for $v\in V$ and $I\subset V$, let $\tilde{\nu}_{I,v}^R=-\partial_v\tilde{\rho}^R_I$ for $v\in I$, and $\tilde{\nu}_{I,v}^R=0$ for $v\notin I$.

By Equation \eqref{open_cover_intersection}, and Property 6 from Proposition \ref{exhaustion_existence}, in the intersection $\mathring{U}_I^{R,\max}\cap \mathring{U}_J^{R,\max}$, $\pi_{I,I\cap J}\circ P_I^R\circ r_I^{\max}$ and $\pi_{J,I\cap J}\circ P_J^R\circ r_J^{\max}$  both agree with $P^R_{I\cap J}\circ r_{I\cap J}^{\max}$. Therefore, in this region,  $\tilde{\rho}_I^R\circ r_I^{\max}$ agrees with $\tilde{\rho}_J^R\circ r_J^{\max}$, as does $\tilde{\nu}_{v,I}^R\circ r_I^{\max}$ with $\tilde{\nu}_{v,J}^R\circ r_J^{\max}$ for $v\in V$.

We can therefore define the functions $\rho^R$ and $\nu_v^R$ for $v\in V$ by letting
\begin{equation}
\rho^R|_{\mathring{U}_I^{R,\max}}=\tilde{\rho}_I^R\circ r_I^{\max}
\end{equation}
and
\begin{equation}
\nu_v^R|_{\mathring{U}_I^{R,\max}}=\tilde{\nu}_{I,v}^R\circ r_I^{\max}.
\end{equation}
These functions satisfy the properties 1-6 by construction.
\end{proof}

\subsection{Outer Caps}

\begin{set}
\label{rhoh}
Suppose we are in Setting \ref{main_setting}.

Suppose $h:\mathbb{R}\rightarrow \mathbb{R}$ is a smooth function such that $h(\rho)=c$ for $\rho\leq 0$, for some $c\in\mathbb{R}$. Then $h\circ \rho^R$ extends smoothly over $L$, and so does $\nu_v^{h,R}=(h'\circ\rho^R)\nu_v^R$.
\end{set}
\begin{defn}[Outer Cap]
Suppose we are in Setting \ref{rhoh}.

Let $\gamma$ be a 1-periodic Hamiltonian orbit of $h\circ \rho^R$.

The outer cap is an element of $\pi_2(M,\gamma)$ defined as follows.

If $\gamma\subset L$, then the outer cap is the constant cap.

Otherwise, if $\rho^R(\gamma)<1$, the outer cap for $\gamma$ is the union of the cylinder swept out by $\gamma$ under the Liouville flow for time $t=-\log(\rho^R(\gamma))$ (so $\psi^t(\gamma)\subset Y^R$), together with a disc $u$ bounding $\psi^t(\gamma)$, such that $diam(u)<inj(M)$.

If $\rho^R(\gamma)\geq 1$, the outer cap for $\gamma$ is a disc $u$ bounding $\gamma$, such that $diam(u)<inj(M)$.
\end{defn}

\begin{lem}
Outer caps exist.
\label{cap_existence}
\end{lem}

\begin{proof}
If $\gamma\subset L$, $d(h\circ \rho^R)=0$, and we can take the cap to be constant.

Otherwise, $\gamma\subset\mathring{U}_I^{R,\max}$ for some $I\subset V$. If $\rho^R(\gamma)<1$, let $t=-\log(\rho^R(\gamma))$. $\psi^t(\gamma)\subset U_I$, and $\psi^t(\gamma)$ is in fact contained in a $\mathbb{R}^I/\mathbb{Z}^I$-orbit of $r_I$.

By property 5 from Lemma \ref{small_nbhds}, $diam(\psi^t(\gamma)))<\frac{1}{2}inj(M)$. This means that $\psi^t(\gamma)$ is contained in a ball of radius $<\frac{1}{2}inj(M)$, so there exists a disc $u$ bounding $\psi^t(\gamma)$, also contained in this ball, and $diam(u)<inj(M)$.

If instead $\rho^R(\gamma)\geq 1$, the same argument applies directly to $\gamma$.
\end{proof}

\begin{lem}
Outer caps are unique up to homotopy.
\end{lem}

\begin{proof}
If $\gamma\subset L$, this is trivially true. 

Otherwise, suppose first that $\rho^R(\gamma)\geq 1$. Any two outer caps for $\gamma$ have diameter $<inj(M)$, so are both contained in a (contractible) ball of radius $<inj(M)$ centred on a point of $\gamma$, and are therefore homotopic.

Now suppose that $\rho^R(\gamma)< 1$. The same argument applies to $\psi^t(\gamma)\subset Y^R$ where $t=-\log(\rho^R(\gamma))$.
\end{proof}

\subsection{Action}
\begin{lem}
\label{action}
Suppose we are in Setting \ref{rhoh}.

Let $\gamma$ be a 1-periodic orbit of $X_{h\circ \rho^R}$.

Then the action with respect to the outer cap $u$ is given by
\begin{equation}
\mathcal{A}_{h\circ\rho^R}(\gamma,u)=h(\rho^R(\gamma))+\sum_{v\in V}\nu_v^{h,R}(\gamma)r_v^{max}(\gamma).
\end{equation}
\end{lem}
\begin{proof}
If $\gamma\subset L$, $u$ is a point, and $\nu_v^{h,R}(\gamma)=0$ for all $v\in V$, so the result is trivial.

Otherwise, $\gamma\subset \mathring{U}_I^{R,\max}$ for some $I\subset V$.

Because $\mathring{U}_I^{R,\max}$ is nonempty, for all $u,v\in I$, either $u\leq v$ or $v\leq u$. We can therefore choose a minimal element $v_0\in I$ such that for all $v\in I$, $v\geq v_0$.

For $v\in I$, let $a_v=\nu_v^{h,R}(\gamma)$.

Let
\begin{equation}
\label{linearisation}
f=\sum_{v\in I} a_vr^{max}_v:U_I^{max}\rightarrow \mathbb{R}.
\end{equation}

$\gamma$ is a 1-periodic orbit of $X_f$.

Suppose first that $\rho^R(\gamma)\geq 1$. Then $\gamma\subset U_I$, and by Properties 3 and 4 from Lemma \ref{small_nbhds}, $\gamma$ is contained in a ball $B^{\frac{1}{2}}$ of radius $\frac{1}{2}inj(M)$, centred at a point in $\mathring{D}_{v_0}$. Let $O\subset U_I$ denote the union of all $\mathbb{R}^I/\mathbb{Z}^I$-orbits generated by $r_I$ intersecting $U_I\cap B^{\frac{1}{2}}$. By Property 4, $O\subset B$, where $B$ is a ball of radius $inj(M)$ centred at the same point as $B^\frac{1}{2}$. Because $B$ is contractible, we have $\omega|_B=d\alpha$ for some $\alpha\in\Omega^1(B)$. 

Let 
\begin{equation}
\overline{\alpha}=\int_{t\in [0,1]^I}(\phi_{r_I}^t)^*\alpha|_{O} dt\in\Omega^1\left(O\right).
\end{equation}
By Equation \eqref{adapted_equivariant}, we have that $\iota_{X_{r_v}}\overline{\alpha}-r_{v}$ is constant on $O$ for $v\in I$. By considering a path of $1$-periodic orbits of $X_{r_v}$ in $O$ converging to a point in $\mathring{D}_{v_0}$, we see that $\iota_{X_{r_v}}\overline{\alpha}=r_v$. By Equation \ref{Ham_pullback_exact}, $\alpha|_O-\overline{\alpha}$ is exact, so
\begin{equation}
\int_u\omega=\int_\gamma \alpha=\int_\gamma \overline{\alpha}=\sum_{v\in I}^ka_vr_{v}(\gamma).
\label{disc}
\end{equation}

Suppose instead that $\rho^R(\gamma)<1$.

Let $t=-\log(\rho^R(\gamma))$. The area of the cylinder swept out by $\gamma$ under the Liouville flow for time $t$ is given by 
\begin{equation}
\int_{\gamma-\psi^t(\gamma)}\theta=\sum_{v\in I}^ka_v(r_{v}^{max}(\gamma)-r_{v}^{\max}(\psi^t(\gamma))).
\label{tube}
\end{equation}

Adding the results of equations \ref{tube} and \ref{disc} applied to $\psi^t(\gamma)$ together yields the required result.
\end{proof}

\subsection{Index}
\begin{lem}
\label{index}
Suppose we are in Setting \ref{rhoh}.

Let $\gamma$ be a 1-periodic orbit of $X_{h\circ \rho^R}$.

Then the Conley-Zehnder index with respect to the outer cap $u$ satisfies
\begin{equation}
\left|CZ_{h\circ\rho^R}(\gamma,u)-2\sum_{v\in V}w_v \nu_v^{h,R}(\gamma) \right|\leq 2n .
\end{equation}
\end{lem}

\begin{proof}
If $\gamma\subset L$, $u$ is a point, and $\nu_v^{h,R}(\gamma)=0$ for all $v\in V$, so the result is trivial.

Otherwise, $\gamma\subset \mathring{U}_I^{R,\max}$ for some $I\subset V$.

The Conley-Zehnder index is invariant under the Liouville flow, so we may assume $\rho^R(\gamma)\geq 1$, $\gamma\subset U_I$, and a disc $u$ bounding $\gamma$ of diameter $<inj(M)$ is an outer cap.

As in the proof of Lemma \ref{action}, let $v_0$ be the minimal element of $I$. The orbit $\gamma$ is contained in a ball $B$ of radius $<inj(M)$ centred on a point in $\mathring{D}_{v_0}$. 

We will first fix a trivialisation of $TM$ over $B$, and because $u\subset B$, we may calculate all Conley-Zehnder indices with respect to this trivialisation. 

Let $a_v=\nu_v^{h,R}(\gamma)$ for $v\in I$, and the function $f$ be as in Equation \eqref{linearisation}.

Fix $x\in \gamma$.

Let
\begin{equation}
A=\frac{d}{dt}(d\phi_{h\circ\rho-f}^t)_x=d(X_{h\circ\rho-f})_x\in End(T_xM).
\end{equation}

We have that 
\begin{equation}
(d\phi_{h\circ \rho}^t)_{\phi^{-t}_f(x)}\circ (d\phi_f^{-t})_x=(d\phi_{h\circ\rho-f}^t)_x=id+tA\in Sp(T_xM).
\end{equation}

Because $id+tA$ is a symplectomorphism, we have that $im(A)$ is isotropic.

Let 
\begin{equation}
K=ker(dr_I)_x\subset T_xM.
\end{equation} 

We have $K\subset ker(A)$ and $im(A)\subset K$, so $im(A)\subset ker(A)$.

Therefore $id+tA$ is a symplectic shear in the sense of Theorem 4.1 of \cite{RS}, so by the normalisation property, $|CZ(id+tA)|=\frac{1}{2}|sign(A)|\leq n$.

Therefore $|CZ(d\phi_{h\circ\rho^R}^t)-CZ(d\phi_f^t)|\leq n$.

By concatenation and homotopy, 
\begin{equation}
CZ(d\phi_f^t)=\sum_{v\in I}CZ(d\phi_{r_v}^{a_vt}).
\end{equation}

To calculate $CZ(d\phi_{r_{v}}^{a_v t})$, we can consider the isotropy representation at a point in $\mathring{D}_{v}\cap B$. 

Suppose $v\in V_{n-l}$, so recall that $w_{v}$ is a sum of $l$ positive integers, coming from a decomposition of the isotropy representation into irreducibles. So we have that $w_{v}=w_1+...+w_l$ for some positive integers $w_j$ and $d\phi_{r_{v}}^{a t}$ is diagonal with entries $e^{2\pi ia_vw_jt}$ for $j=1,...,l$.

Since $CZ(e^{2\pi iat})=2\lceil a \rceil$, and 
\begin{equation}
\sum_{j=1}^la_vw_j\leq \sum_{j=1}^l \lceil a_vw_j \rceil \leq l+\sum_{j=1}^l  a_vw_j,
\end{equation}
by the product property, we have 
\begin{equation}
2w_{v}a_v \leq CZ(d\phi_{r_{v}}^{a_vt})\leq 2w_{v}a_v +n.
\end{equation}

Therefore, we have

\begin{equation}
\left|CZ(d\phi_{h\circ\rho}^t)-2\sum_{v\in I} w_{v} a_v \right|\leq 2n.
\end{equation}

\end{proof}

\subsection{Main Theorem}
\begin{set}
\label{monotone_setting}
Suppose we are in Setting \ref{main_setting}, and suppose additionally that $(M,\omega)$ is positively monotone, i.e.
\begin{equation}
[\omega]=2\kappa c_1(TM)
\end{equation}
for some $\kappa>0$.

Let $\kappa_v=\kappa\lambda_v$ for $v\in V$, and
\begin{equation}
\sigma_{crit}=\max\left(0,\max_{v\in V}\left(\frac{\lambda_v-2w_v}{\lambda_v}\right)\right).
\end{equation}

To define superheaviness, we associate to any time-independent Hamiltonian $H:M\rightarrow \mathbb{R}$ a partial symplectic quasi-state,
\begin{equation}
\zeta(H)=\lim_{\ell\to +\infty}\frac{c([M],\ell H)}{\ell}
\end{equation}
as in Section 3.5 of \cite{EP}, where the spectral number $c([M],H)$ is defined using the PSS isomorphism as in Section 3.4 of \cite{EP}.

Recall that a subset $K\subset M$ is superheavy if $\zeta(H)\leq \sup_K H$ for any Hamiltonian $H:M\rightarrow \mathbb{R}$.
\end{set}

\begin{lem}
\label{main_estimate}
There exists a continuous, increasing function 
\begin{equation}
\sigma:[0,R_0)\rightarrow [\sigma_{crit},\infty)
\end{equation}
such that $\sigma(0)=\sigma_{crit}$, and for $R\in (0,R_0)$,
\begin{equation}
\rho^R+ \sum_{v\in V}(r_v^{\max}-2\kappa w_v)\nu_v^R\leq \sigma(R). 
\label{sigma_ineq}
\end{equation}
\end{lem}
\begin{proof}
Let 
\begin{equation}
\sigma(R)=\max\left(0,\max_{v\in V}\left(\frac{\lambda_v-2w_v}{\lambda_v-\kappa^{-1}\epsilon_v(R)}\right)\right).
\end{equation}

Expressing $\rho^R$ in terms of the functions $\nu_v^R$ and $r_v^{\max}$ for $v\in V$, we have
\begin{equation}
\rho^R=\sum_{v\in V}(\kappa\lambda_v-r_v^{\max})\nu_v^R.
\label{rho_nu}
\end{equation}

The expression on the left of \eqref{sigma_ineq} is constant along flowlines of $Z$, so it suffices to check Equation \eqref{sigma_ineq} at a point $x\in Y^R$. We have $\rho^R(x)=1$, and for $v\in V$, either $\nu_v^R(x)=0$ or $r_v^{\max}(x)<\epsilon_v(R)$.

Substituting Equation \eqref{rho_nu} into the expression gives
\begin{equation}
\begin{split}
\rho^R+\sum_{v\in V}(r_v^{\max}-2\kappa w_v)\nu_v^R&=\kappa\sum_{v\in V}(\lambda_v-2w_v)\nu_v^R.
\\&=\sum_{v\in V}\frac{\lambda_v-2w_v}{\lambda_v-\kappa^{-1}r_v^{\max}}(\kappa\lambda_v-r_v^{\max})\nu_v^R
\\&<\max_{v\in V}\left(\frac{\lambda_v-2w_v}{\lambda_v-\kappa^{-1}r_v^{\max}}\right)\rho^R
\\&<\max_{v\in V}\left(\frac{\lambda_v-2w_v}{\lambda_v-\kappa^{-1}\epsilon_v(R)}\right)
\\&<\sigma(R).
\end{split}
\end{equation}

\end{proof}
\begin{lem}
\label{h_rho_estimate}
Let $\sigma$ be as in Lemma \ref{main_estimate} and let $h:\mathbb{R}\rightarrow \mathbb{R}$ be a smooth function such that $h''\geq 0$ and $h(\rho)=h_0$ for $\rho<\sigma(R)$, for some $h_0\in\mathbb{R}$.

Then
\begin{equation}
\zeta(h\circ \rho^R)\leq h_0.
\end{equation}
\end{lem}
\begin{proof}

As in the proof of Proposition 9.1 of \cite{EP}, we will bound $\zeta(h\circ\rho^R)$ by estimating $c([M],H_\ell)$ for a sequence of small regular perturbations $H_\ell$ of $\ell h\circ\rho^R$ for $\ell \in\mathbb{N}$.

For any Hamiltonian $H$ and any 1-periodic orbit $\gamma$ of $X_H$, let
\begin{equation}
\label{mixed_index}
D_H(\gamma)=\mathcal{A}_H(\gamma,u)-\kappa CZ_H(\gamma,u)
\end{equation}
for any $u\in\pi_2(M,\gamma)$. This quantity is independent of $u$ because $(M,\omega)$ is monotone.

As in the aforementioned proof in \cite{EP}, as $\ell\to \infty$,
\begin{equation}
c([M],H_\ell)=D_{H_\ell}(\tilde{\gamma}_\ell)+O(1)
\end{equation}
where $\tilde{\gamma}_\ell$ is some 1-periodic orbit of $X_{H_\ell}$. By Proposition 9.2 of \cite{EP}, provided the perturbations $H_\ell$ are sufficiently $C^\infty$-close to the Hamiltonians $\ell h\circ\rho^R$, we have
\begin{equation}
c([M],H_\ell)=D_{\ell h\circ\rho^R}(\gamma_\ell)+O(1)
\end{equation}
where $\gamma_\ell$ is some 1-periodic orbit of $X_{\ell h\circ\rho^R}$.

By Lemmas \ref{action} and \ref{index}, 
\begin{equation}
\label{D_estimate}
c([M],H_\ell)=\ell h(\rho^R(\gamma_\ell))+\ell \sum_{v\in V}\left(r_v^{\max}(\gamma_\ell)-2\kappa w_v\right)\nu_v^{h,R}(\gamma_\ell)+O(1).
\end{equation}

By convexity of $h$,
\begin{equation}
h\circ \rho^R \leq h_0+(h'\circ\rho^R)\left(\rho^R-\sigma(R)\right),
\end{equation}
and by Lemma \ref{main_estimate}, 
\begin{equation}
h\circ \rho^R \leq h_0-\sum_{v\in V}(r_v^{\max}-2\kappa w_v)\nu_v^{h,R}.
\label{h_convexity}
\end{equation}

Substituting Equation \eqref{h_convexity} into Equation \ref{D_estimate}, we have
\begin{equation}
c([M],H_\ell)\leq \ell h_0 +O(1).
\end{equation}
Taking the appropriate limit gives the required result.
\end{proof}

\begin{prop}
\label{superheavy_subset}
The subset 
\begin{equation}
K_{crit}=(\rho^0)^{-1}[0,\sigma_{crit}]
\end{equation}
is superheavy.
\end{prop}
\begin{proof}
This proof follows the same idea as that of Theorem 1.9 of \cite{EP}, in that we will bound an arbitrary Hamiltonian by one over which we have more control.

Let $H:M\rightarrow \mathbb{R}$ be a smooth Hamiltonian. It suffices to prove that $\zeta(H)\leq \sup_{K_{crit}} H$.

For any $\delta>0$, let $h_0=\sup_{K_{crit}}H+\delta$. By continuity we can choose $R\in (0,R_0)$ such that $H|_{(\rho^R)^{-1}[0,\sigma(R)]}<h_0$. 

We can then choose some real function $h:\mathbb{R}\rightarrow \mathbb{R}$ such that $h\circ \rho^R\geq H$, $h''\geq 0$, and $h|_{[0,\sigma(R)]}=h_0$.

Now by Lemma \ref{h_rho_estimate}, and monotonicity of $\zeta$,
\begin{equation}
\zeta(H)\leq\zeta(h\circ\rho^R)\leq h_0.
\end{equation}
Since this is true for all $\delta>0$, we have that $\zeta(H)\leq \sup_{K_{crit}}H$ as required.
\end{proof}

\begin{cor}
\label{superheavy_skeleton}
If $\lambda_v\leq 2w_v$ for all $v\in V$, then $L$ is superheavy.
\end{cor}
\begin{proof}
We have $\sigma_{crit}=0$, so by Proposition \ref{superheavy_subset}, $(\rho^0)^{-1}(0)=L$ is superheavy.
\end{proof}

\section{Algebraic Varieties}
\label{s6}

\subsection{Algebraic Varieties}

\begin{set}
\label{algebraic_setting}
Let $M$ be an $n$-dimensional complex projective variety, and $D$ an effective ample $\mathbb{Q}$-divisor on $M$, such that
\begin{equation}
[D]=-2K_M.
\end{equation}

Let $N\in\mathbb{N}$ be a positive integer such that $ND$ is a $\mathbb{Z}$-divisor. Let $\mathcal{L}$ be a line bundle corresponding to $ND$. Equip $\mathcal{L}$ with a hermitian metric, and let $F$ denote the curvature form of the Chern connection on $\mathcal{L}$.

For any $\kappa>0$, 
\begin{equation}
\omega=-\frac{i\kappa}{2N\pi}F
\end{equation}
is a symplectic form on $M$, and 
\begin{equation}
[\omega]=2\kappa c_1(TM).
\end{equation}

The complement $X=M\setminus D$ is an affine variety. For any global section $s$ of $\mathcal{L}$, 
\begin{equation}
\theta=-\frac{\kappa}{2N\pi}d^c\log ||s||
\end{equation}
is a primitive for $\omega|_X$, making $X$ a finite type convex symplectic manifold.
\end{set}

\begin{defn}[Algebraic Divisors as Stratified Symplectic Divisors]
Suppose we are in setting \ref{algebraic_setting}.

Let $V_D$ denote the minimal set of irreducible subvarieties of $M$ such that:
\begin{itemize}
\item $V_D$ contains the irreducible components of $D$
\item if $Y_1,Y_2\in V_D$, then the irreducible components of $Y_1\cap Y_2$ are in $V_D$,
\item for any $Y\in V_D$, all irreducible components of the singular locus of $Y$ are in $V_D$.
\end{itemize}

We give $V_D$ the structure of a stratified poset of height $n-1$ by setting $Y\leq Z$ whenever $Y\subset Z$, and letting $(V_D)_l=\{Y\in V_D| dim(Y)=l\}$.

We give $V_D$ the structure of a stratified symplectic subvariety in $(M,\omega)$ by letting $\mathring{D}_Y$ denote the smooth part of $Y$.

We have that $|V_D|=D$ as a subset of $M$, and $(V_D)_{n-1}$ is the set of irreducible components of $D$. 
\end{defn}
\begin{lem}
\label{hypersurf_adapted}
Suppose we are in Setting \ref{algebraic_setting}.
 
Fix $v\in (V_D)_{n-1}$ and let $r_v:U_v\rightarrow \mathbb{R}_{\geq 0}$ be a radial Hamiltonian for $v$.

Then $\theta$ is weakly adapted at $v$.
\end{lem}

\begin{proof}
Fix $x\in \mathring{D}_v$.

Choose local holomorphic coordinates $(z_1,...,z_n):U\rightarrow \mathbb{C}^n$ near $x$, such that $x=(0,...,0)$, the plane $\{z_2=...=z_n=0\}$ intersects $\mathring{D}_v$ symplectically orthogonally at $x$, $\mathring{D}_v\cap U=\{z_1=0\}$, and $(dz_1)_x:T_x\mathring{D}_v^\omega\rightarrow \mathbb{C}$ is a symplectomorphism, with respect to the standard symplectic form on $\mathbb{C}$. 

The isotropy representation $d\phi_{r_v}^t$ on $T_x\mathring{D}_v^\omega$ induces a symplectic representation of $\mathbb{R}/\mathbb{Z}$ on $\mathbb{C}$ via $(dz_1)_x$, which is related to the standard representation by some $A\in Sp(2)$. With respect to the coordinates $(z_1,...,z_n)$, $X_{r_v}=(A^*(2\pi iz_1),0,...,0)+O(||z||^2)$. Near $x$, $s|_U=z_1^{N\lambda_v}f$, where $f$ is holomorphic and $f\neq 0$, so $\theta=-\frac{\kappa\lambda_v}{2\pi}d^c\log |z_1|+O(1)$.

Let $z$ be a coordinate on $\mathbb{C}$. Then on $\mathbb{C}^*$,
\begin{equation}
\iota_{A^*(iz)}d^c\log |z|=\frac{\langle z,iAiA^{-1}z\rangle}{2|z|^2} =-\frac{\langle (iA)^Tz,iA^{-1}z\rangle}{2|z|^2}=-\frac{|iA^{-1}z|^2}{2|z|^2}<-B
\end{equation}
for some constant $B>0$, so near $x$, 
\begin{equation}
\iota_{X_{r_v}}\theta<-\kappa \lambda_v B +O(||z||)
\end{equation} 
so, on some neighbourhood of $x$, $\iota_{X_{r_v}}\theta<0$. This is true for all $x\in\mathring{D}_v$, and therefore on a stratum neighbourhood for $v$ as required.
\end{proof}

\subsection{Quasihomogeneous Singularities}
\begin{defn}[Quasihomogeneous Polynomial]
A polynomial $f\in \mathbb{C}[z_1,...,z_n]$ is quasihomogeneous with weight system $(a_1,...,a_n,\lambda)\in\mathbb{Z}_{\geq 0}^{n+1}$ if for any $t\in\mathbb{C}^*$, 
\begin{equation}
f(t^{a_1}z_1,...,t^{a_n}z_n)=t^\lambda f(z_1,...,z_n).
\end{equation}
\end{defn}
\begin{lem}
\label{qhsings}
Suppose we are in Setting \ref{algebraic_setting}, $v\in V_D$, and $\mathring{D}_v$ is contained in a local analytic chart $(z_1,...,z_n):U\rightarrow\mathbb{C}^n$ such that $\mathring{D}_v=\{z_{k+1}=...=z_n=0\}$, and $s|_U=f\in\mathbb{C}[z_1,...,z_n]$, where $f$ is quasihomogeneous with weight system $(a_1,...,a_k,0,...,0,\lambda)$, where $a_i>0$ for $i=1,...,k$ and $d>0$.

Then there exists a radial Hamiltonian $r_v:U_v\rightarrow \mathbb{R}_{\geq 0}$ for $v$, where $U_v\subset U$, such that the action generated by $r_v$ factors through the standard $(\mathbb{C}^*)^n$-action on $\mathbb{C}^n$, the total weight of $r_v$ is
\begin{equation}
w_v=\sum_{i=1}^ka_i,
\end{equation}
$\theta$ is weakly adapted to $r_v$, and the action of $r_v$ with respect to $\theta$ is $\kappa\lambda_v$ where
\begin{equation}
\lambda_v=\frac{\lambda}{N}.
\end{equation}
\end{lem}
\begin{proof}
First, choose a smaller stratum neighbourhood $U_v\subset U$ which is a union of slices of the form $\Delta^k\times\mathbb{C}^{n-k}$ where $\Delta$ is a disc centred at $0$. 

There is an algebraic $\mathbb{R}/\mathbb{Z}$-action on $U_v$, given by
\begin{equation}
t\mapsto diag(e^{2\pi a_1 it},...,e^{2\pi a_k it},1,...,1).
\end{equation}

The function $\frac{t^*f}{f}$ is holomorphic on $X\cap U_v$ and extends to $U_v$, so $t^*\omega-\omega=dd^c\log\left|\frac{t^*f}{f}\right|=0$ because $\log \left|\frac{t^*f}{f}\right|$ is pluriharmonic, so the action is symplectic. Because $U_v$ deformation retracts onto $\mathring{D}_v$, we may assume that the action is generated by some Hamiltonian $r_v:U_v\rightarrow\mathbb{R}$, and $r_v|_{\mathring{D}_v}=0$.

By considering any point in $\mathring{D}_v$, where $z_i=0$ for $i\geq k+1$, we see that the isotropy representation of the action is $t\mapsto diag(e^{2\pi a_1 it},...,e^{2\pi a_k it})$, so the total weight of $r_v$ is indeed $\sum_{i=1}^ka_i$ as required. Furthermore, the fact that $a_i>0$ for all $i$ tells us that the fixed locus $\mathring{D}_v$ is a local minimum for $r_v$, so, possibly shrinking $U_v$, we may assume that $r_v$ is in fact a map $r_v:U_v\rightarrow \mathbb{R}_{\geq 0}$, and $r_v^{-1}(0)=\mathring{D}_v$.

With respect to the coordinates $(z_1,...,z_n)$, $iX_{r_v}=-2\pi (a_1z_1,...,a_kz_k,0,...,0)$, and $\log ||s||=\log |f|+\log ||1||$, so
\begin{equation}
\begin{split}
\iota_{X_{r_v}}\theta&=-\frac{\kappa}{2N\pi}\mathcal{L}_{iX_{r_v}}\log |s|
\\&=\frac{\kappa}{2N\pi}\mathcal{L}_{iX_{r_v}}\log |f|+O(||\hat{z}||)
\\&=\frac{\kappa}{2N\pi}\frac{d}{dt}|_{t=0}\log|f(e^{-2\pi a_1 t}z_1,...,e^{-2\pi a_k t}z_k,z_{k+1},...,z_n)|+O(||\hat{z}||)
\\&=-\frac{\kappa}{2N\pi}\frac{d}{dt}|_{t=0}\log|e^{-2\pi \lambda t}f(z_1,...,z_n)|+O(||\hat{z}||)
\\&=-\frac{\lambda\kappa}{N}+O(||\hat{z}||),
\end{split}
\end{equation}
where $\hat{z}=(z_{k+1},...,z_n)\in\mathbb{C}^{n-k}$, so $\iota_{X_{r_v}}\theta<0$ in a neighbourhood of $\mathring{D}_v$, and by considering a path of orbits converging to a point in $\mathring{D}_v$, we see that the action is $\frac{\kappa \lambda}{N}$ as required.

\end{proof}

\begin{cor}
\label{qhsch}
If $n=2$, and all singular points of $D$ are quasihomogeneous, then $V_D$ admits a system of commuting Hamiltonians, to which $\theta$ is weakly adapted.
\end{cor}

\begin{proof}
We have that $(V_D)_0$ corresponds to the set of singular points of $D$, and $(V_D)_1$ to the irreducible components.
By Lemma \ref{qhsings}, we can choose a radial Hamiltonian for each $v\in (V_D)_0$. By \ref{hypersurf_ham}, we can extend this to a system of commuting Hamiltonians for $V_D$. By \ref{qhsings}, $\theta$ is adapted at each $v\in (V_D)_0$, and by \ref{hypersurf_adapted}, $\theta$ is weakly adapted at each $v\in (V_D)_1$.
\end{proof}

\section{Examples}
\subsection{Klein bottle in $\mathbb{CP}^1\times\mathbb{CP}^1$}

\begin{lem}
\label{Klein_bottle_skeleton}
Endow $M=\mathbb{CP}^1\times\mathbb{CP}^1$ with the Fubini-Study metric, and let $[x_i:y_i]$ be homogeneous coordinates on the $i$-th factor.

Let
\begin{equation}
s=x_2y_2(x_1^2y_2+y_1^2x_2)^2\in\Gamma(\mathcal{O}(4,4))
\end{equation}

Let $D=s^{-1}(0)$, and $X=M\setminus D$.

The skeleton of $(X,d^c\log ||s||)$ is the Lagrangian Klein bottle $K$ described in Theorem \ref{Klein_bottle}.
\end{lem}

\begin{proof}
Let $(p_i,\theta_i)$ be action-angle coordinates on the $i$-th factor, as described in Theorem \ref{Klein_bottle}. With respect to the Fubini-Study metric,
\begin{equation}
\begin{split}
||s||^2=\left(\frac{1}{4}-p_2^2\right)\left(\left(\frac{1}{2}-p_1\right)^2\left(\frac{1}{2}+p_2\right)+\left(\frac{1}{2}+p_1\right)^2\left(\frac{1}{2}-p_2\right)\right.\\\left.+2\left(\frac{1}{4}-p_1^2\right)\sqrt{\frac{1}{4}-p_2^2}\cos(2\theta_1+\theta_2)\right)^2.
\end{split}
\end{equation}
The primitive $d^c\log ||s||$ is invariant under the Hamiltonian $\mathbb{R}/\mathbb{Z}$-action 
\begin{equation}
t\cdot \left([x_1:y_1],[x_2:y_2]\right)=\left([x_1:e^{2\pi i t}y_1],[x_2:e^{-4\pi it}y_2]\right),
\end{equation}
which is generated by $p_1-2p_2$. By Lemma \ref{skeleton_barycentre}, the skeleton is contained in a fibre of this map, which, by symmetry, is $\{p_1=2p_2\}$. On restriction to this fibre, 
\begin{equation}
||s||=\frac{1}{2}\sqrt{1-p_1^2}\left(\frac{1}{4}+\left(\frac{1}{4}-p_1^2\right)\sqrt{1-p_1^2}\cos\left(2\theta_1+\theta_2\right)\right).
\end{equation}

The critical points of $-\log ||s||$ consist three critical $\mathbb{R}/\mathbb{Z}$-orbits: the global minima $\{p_1=p_2=0, 2\theta_1+\theta_2=0\}$ and the saddle points $\{p_1=\pm\frac{1}{2}, p_2=\pm\frac{1}{4}\}$.

Let
\begin{equation}
r=\sqrt{\frac{1}{4}-p_1^2}
\end{equation}
and
\begin{equation}
x=r\cos(2\theta_1+\theta_2).
\end{equation}
In the complement of the critical points in the fibre $\{p_1=2p_2\}$, 
\begin{equation}
\nabla ||s||\cdot r=-2\pi r \left(p_1^2\sqrt{1-p_1^2}-2p_1^2(1-p_1^2)\cos(2\theta_1+\theta_2)\right)
\end{equation}
and
\begin{equation}
\nabla ||s||\cdot x=-2\pi x\left(p_1^2\sqrt{1-p_1^2}+\left(\frac{5}{4}-4p_1^2+2p_1^4\right)\cos(2\theta_1+\theta_2)\right),
\end{equation}
In particular, in the region $-\frac{\pi}{2}\leq 2\theta_1+\theta_2\leq\frac{\pi}{2}$, $-\nabla \log ||s||\cdot |x|>0$, and outside this region, $-\nabla \log ||s||\cdot r>0$. 

This implies that the region $\{x=0,-\frac{\pi}{2}\leq 2\theta_1+\theta_2\leq\frac{\pi}{2}\}$, which is equal to $\{p_1=2p_2,2\theta_1+\theta_2\}$, is preserved by the gradient flow.

Furthermore, for any $\epsilon>0$, $-\nabla \log ||s||$ is outward pointing along the boundary of the region
\begin{equation}
\{r\leq \epsilon\}\cup \left\{|x|\leq \epsilon, -\frac{\pi}{2}\leq 2\theta_1+\theta_2\leq\frac{\pi}{2}\right\}\subset \{p_1=2p_2\},
\end{equation}
which contains all the critical points of $-\log ||s||$. The skeleton is therefore contained in the intersection of all these regions for $\epsilon>0$, so the skeleton is exactly $\{p_1=2p_2,2\theta_1+\theta_2\}$.
\end{proof}

\begin{cor}[Proof of Theorem \ref{Klein_bottle}]
The Lagrangian Klein bottle  $K\subset \mathbb{CP}^1\times\mathbb{CP}^1$ is superheavy.
\end{cor}

\begin{proof}
Let $X$, $D$ and $s$ be as in Lemma \ref{Klein_bottle_skeleton}.

Let $H_0=\{y_2=0\}$, $H_\infty=\{x_2=0\}$, and $Q=\{x_1^2y_2+x_2^2y_1=0\}$. We have that $D=H_0+H_\infty+2Q$, and
\begin{equation}
[D]=-2K_{\mathbb{CP}^1\times\mathbb{CP}^1}.
\end{equation}

The irreducible components $H_0$, $H_\infty$ and $Q$ are smooth, so $(V_D)_0$ consists of the intersection points $p_0=\left([0:1],[1:0]\right)$ and $p_\infty=\left([1:0],[0:1]\right)$.

In local coordinates near both of these points, $s$ has the form $z_1(z_1+z_2^2)^2$, which is quasihomogeneous with weight system $(2,1,6)$.

By Corollary \ref{qhsch}, there exists a system of commuting Hamiltonians for $V_D$, to which $d^c\log ||s||$ is weakly adapted. This means that we can choose an adapted one-form which also has $K$ as a skeleton, without changing the action of any of the radial Hamiltonians.

By Lemma \ref{qhsings}, for $Y=p_0,p_\infty$, we have $\lambda_Y=6$ and $w_Y=2+1=3$, so $\lambda_Y=2w_Y$.

Because for $Y=L_0,L_\infty,C$, $\lambda_Y$ is equal to the coefficient of $Y$ in $D$ which is either $1$ or $2$, and $w_Y=1$, $\lambda_Y\leq 2w_Y$.

By Corollary \ref{superheavy_skeleton}, the Lagrangian skeleton $K$ is superheavy.
\end{proof}
\subsection{Intersecting $\mathbb{RP}^2$s}
\begin{lem}
\label{arcs_approximation}
Take $M=\mathbb{CP}^2$, with the Fubini-Study symplectic form.
Let $[z_0:z_1:z_2]$ be homogeneous coordinates on $M$.

Let $D_\pm=\{z_1z_2\pm z_0^2=0\}\subset \mathbb{CP}^2$.

Fix $a,b\in\mathbb{N}$, and let $D=aD_++bD_-$.

Let $L_\alpha\subset \mathbb{CP}^2$ for $\alpha\in \mathbb{R}/2\pi \mathbb{Z}$ be the Lagrangian $\mathbb{RP}^2$s described in Theorem \ref{RP2s}.

Let $U\subset\mathbb{CP}^2$ be any neighbourhood of $L_0\cup L_{\frac{2\pi a}{a+b}}$. Then there exists a Hamiltonian isotopy $\phi^t$ of $M$ such that $\phi^1(U)$ contains the skeleton of $(X,d^c\log ||s||)$.
\end{lem}

\begin{proof}
Let
\begin{equation}
s_\pm =z_1z_2\pm z_0^2\in \mathcal{O}(2)
\end{equation}
and
\begin{equation}
s=s_+^as_-^b \in \mathcal{O}(2(a+b)),
\end{equation}
which has $D$ as its zero-divisor.

Let $(p_i,\theta_i)$ for $i=1,2$ be the standard action-angle coordinates. With respect to the Fubini-Study metric, 
\begin{equation}
\begin{split}
||s_\pm||^2=p_1p_2+(1-p_1-p_2)^2\mp 2(1-p_1-p_2)\sqrt{p_1p_2}\cos(\theta_1+\theta_2)
\end{split}.
\end{equation}

The primitive $d^c\log ||s||$ is invariant under the Hamiltonian circle action
\begin{equation}
t\cdot [z_0:z_1:z_2]=[z_0:e^{2\pi it}z_1:e^{-2\pi it}z_2],
\end{equation}
which is generated by $p_1-p_2$, so by Lemma \ref{skeleton_barycentre}, the skeleton is contained in a fibre, which, by symmetry, is $\{p_1=p_2\}$. We now consider the restriction of the gradient flow to this fibre.

Let $S=\{p_1=p_2\}/(\mathbb{R}/\mathbb{Z})$ denote the symplectic reduction of the fibre, which is a sphere with two singular points, $\{p,q\}$ and $\pi:\{p_1=p_2\}\rightarrow S$ the quotient map.

The curves $D_+$ and $D_-$ intersect the fibre along the orbits $\{p_1=p_2=\frac{1}{3},\theta_1+\theta_2=0\}$ and $\{p_1=p_2=\frac{1}{3},\theta_1+\theta_2=\pi\}$ respectively. Let $m_\pm\in S$ denote the images of these orbits under $\pi$.

The function $||s||^2|_{\{p_1=p_2\}}$ descends to $S$, and has global maxima $p$ and $q$, global minima at $m_\pm$, and two other critical points. The skeleton is therefore the preimage  under $\pi$ of two embedded arcs $\beta_1,\beta_2$ in $S$ connecting the points $p$ and $q$. 

On restriction to the fibre $\{p_1=p_2\}$,
\begin{equation}
\begin{split}
\nabla ||s||^2\cdot (\theta_1+\theta_2)&=\frac{||s||^2}{||s_+||^2||s_-||^2}\left(a||s_-||^2-b||s_+||^2\right)sin(\theta_1+\theta_2)
\end{split}
\end{equation}
so in the case $a=b$, the skeleton is exactly $\{p_1=p_2, \theta_1+\theta_2=\pm \frac{\pi}{2}\}=L_{-\frac{\pi}{2}}\cup L_{\frac{\pi}{2}}$.

In general, the arcs $\beta_1$ and $\beta_2$ partition $S$ into two discs, each of which contains one of $m_\pm$. Because the skeleton is monotone, by considering the intersection numbers with the curves $D_\pm$, we see that the areas of the discs must be in the ratio $a:b$.

The union $L_0\cup L_{\frac{2\pi a}{a+b}}$ is also a preimage of two arcs $\gamma_1,\gamma_2$ from $p$ to $q$, which also partition $S$ into two discs with areas in the ratio $a:b$. Let $U$ be any neighbourhood of $L_0\cup L_{\frac{2\pi a}{a+b}}$. The neighbourhood $U$ contains $\pi^{-1}(T)$, where $T$ is a neighbourhood of $\gamma_1\cup \gamma_2$ in $S$.

There exists a Hamiltonian isotopy $\phi^t$ of $S$, supported away from $\{p,q\}$, such that $\phi^t(T)$ contains $\beta_1\cup \beta_2$. We can lift $\phi^t$ via $\pi$ to a Hamiltonian isotopy $\phi^t$ of $\mathbb{CP}^2$, supported in a neighbourhood of the fibre $\{p_1=p_2\}$, such that $\phi^t(U)$ contains the skeleton.
\end{proof}
\begin{lem}
\label{arcs_superheavy}
If $\frac{a}{a+b}\in\left[\frac{1}{3},\frac{2}{3}\right]$, then the skeleton of $(X,d^c\log ||s||)$, as described in Lemma \ref{arcs_approximation} is $SH$-full.
\end{lem}
\begin{proof}
Let $D$ be as in Lemma \ref{arcs_approximation}.
\begin{equation}
-2K_{\mathbb{CP}^2}=\frac{3}{a+b}[D].
\end{equation}
The irreducible components $D_{\pm}$ are smooth, and intersect at the points $[0:1:0]$ and $[0:0:1]$. Near both intersection points, there exist local coordinates such that $s$ has the form $(x+y^2)^a(x-y^2)^b$, which is quasihomogeneous with weight system $(2,1,2(a+b))$.

By Corollary \ref{qhsch}, $V_D$ admits a system of commuting Hamiltonians with 
\begin{equation}
\sigma_{crit}=\max\left(0,\frac{a-2b}{3a},\frac{b-2a}{3b}\right).
\end{equation}
By our condition on $a$ and $b$, $\sigma_{crit}=0$, so by Corollary \ref{superheavy_skeleton}, the skeleton is indeed $SH$-full.
\end{proof}
\begin{cor}[Proof of Theorem \ref{RP2s}]
Let $L_\alpha\subset \mathbb{CP}^2$ for $\alpha\in \mathbb{R}/2\pi \mathbb{Z}$ be as in Theorem \ref{RP2s}. For any $\alpha\in\left[\frac{2\pi}{3},\frac{4\pi}{3}\right]$, $L_0\cup L_\alpha\subset\mathbb{CP}^2$ is superheavy.
\end{cor}

\begin{proof}
Let $U\subset \mathbb{CP}^2$ be any neighbourhood of $L_0\cup L_\alpha$ for some $\alpha\in\left[\frac{2\pi}{3},\frac{4\pi}{3}\right]$. 

Choose $\frac{a}{a+b}\in\mathbb{Q}\cap \left[\frac{1}{3},\frac{2}{3}\right]$ (where $a,b\in\mathbb{N}$) sufficiently close to $\frac{\alpha}{2\pi}$ such that $L_0\cup L_{\frac{2\pi a}{a+b}}$ is contained in $U$. Let $X$ and $||s||$ be as in Lemma \ref{arcs_approximation}. By the same Lemma, there exists a Hamiltonian isotopy $\phi^t$ of $\mathbb{CP}^2$ such that $\phi^1(U)$ contains the skeleton of $(X,d^c\log ||s||)$. 

By Lemma \ref{arcs_superheavy}, the skeleton, and therefore $\phi^1(U)$ is superheavy. This property is invariant under Hamiltonian isotopy, so $U$ is also superheavy.

This is true for any such $U$, so $L_0\cup L_\alpha$ is also superheavy.
\end{proof}
\subsection{Intersecting $\mathbb{RP}^3$s}
\begin{lem}
\label{RP3_skeleton}
Take $M=\mathbb{CP}^3$, with the Fubini-Study symplectic form.
Let $[z_0:z_1:z_2:z_3]$ be homogeneous coordinates on $M$.

Let
\begin{equation}
s=z_1^2z_2^2-z_3^2z_0^2\in \mathcal{O}(-4)
\end{equation}

Let $D=s^{-1}(0)$ and $X=M\setminus D$.

The skeleton of $(X,d^c\log ||s||)$ is the union of two embedded Lagrangian $\mathbb{RP}^3$s described in Theorem \ref{RP3s}.
\end{lem}

\begin{proof}
Let $p_i$ for $i=0,...,3$ and $\varphi$ be as in Theorem \ref{RP3s}.

The primitive $d^c\log||s||$ is invariant under the hamiltonian $\mathbb{R}^2/\mathbb{Z}^2$-action
\begin{equation}
(t_1,t_2)\cdot [z_0:z_1:z_2:z_3]=[e^{-2\pi it_2}z_0:e^{2\pi it_1}z_1:e^{-2\pi i t_1}z_2:e^{2\pi it_2}z_3],
\end{equation}
which is generated by $(p_1-p_2,p_3-p_0)$, so by Lemma \ref{skeleton_barycentre}, the skeleton is contained in a fibre, which, by symmetry, is $\{p_1=p_2,p_3=p_0\}$. We now consider the restriction of the gradient flow to this fibre.

With respect to the Fubini-Study metric,
\begin{equation}
||s||^2=p_1^2p_2^2+p_3^2p_0^2-2p_1p_2p_3p_0\cos (2\varphi).
\end{equation}

The critical points of $-\log ||s||$ consist of four $\mathbb{R}^2/\mathbb{Z}^2$-orbits, the saddle points $\{p_1=p_2=0,p_3=p_0=\frac{1}{2}\}$ and $\{p_1=p_2=\frac{1}{2},p_3=p_0=0\}$, and the global minima $\{p_1=p_2=p_3=p_0=\frac{1}{4}, \varphi=\pm\frac{\pi}{2}\}$.

In the region $\{p_1=p_2\neq 0,\frac{1}{2}, p_3=p_0\neq 0,\frac{1}{2},\varphi\neq 0,\pm\frac{\pi}{2},\pi\}$,
\begin{equation}
\nabla ||s||^2 \cdot \varphi=-\pi p_1(1-2p_1)\sin(2\varphi)
\end{equation}
so $-\nabla \log ||s||\cdot \cos(2\varphi)\leq 0$. The regions  $\{p_1=p_2,p_3=p_0,\varphi=\pm\frac{\pi}{2}\}$ are therefore preserved by the gradient flow. Furthermore, any gradient flowline which converges to a saddle point must intersect one of these regions, so the skeleton is a union of exactly these two regions, as required.
\end{proof}

\begin{lem}[Proof of Theorem \ref{RP3s}]
The intersecting pair of $\mathbb{RP}^3$s in $\mathbb{CP}^3$ described in Lemma \ref{RP3s} is superheavy.
\end{lem}

\begin{proof}
Let $X$, $D$ and $s$ be as in Lemma \ref{RP3_skeleton}.

Let $D_\pm=\{z_1z_2\pm z_3z_0=0\}$. We have that $D=D_++D_-$, and
\begin{equation}
2[D]=-2K_{\mathbb{CP}^2}.
\end{equation}

$D_\pm$ are smooth hypersurfaces. Their intersection $D_+\cap D_-$ consists of (smooth) lines of the form $\{z_i=z_j=0\}$ for $i=1,2$ and $j=3,0$, so $(V_D)_1$ consists of these four lines, and $(V_D)_0$ consists of their four pairwise intersection points, $\{z_i=z_j=z_k=0\}$ for $i,j,k$ distinct.

In accordance with Lemma \ref{qhsings}, the function $p_i+p_j$ is a radial hamiltonian for the line $\{z_i=z_j=0\}$, with weight $2$ and action $4$, and $p_i+2p_j+p_k$ is a radial Hamiltonian for the point $\{z_i=z_j=z_k=0\}$ with weight $4$ and action $8$, which form a system of commuting Hamiltonians for $(V_D)_{\leq 1}$, where either $\{i,k\}=\{1,2\}$ or $\{i,k\}=\{3,0\}$.

By Proposition \ref{hypersurf_ham}, we can extend these choices to a system of commuting Hamiltonians for $V_D$. The radial Hamiltonians for $D_\pm$ have weight $1$ and normalised action $2$, which is the coefficient of $D_\pm$ in $2D$.

For each $Y\in V_D$, $\lambda_Y= 2w_Y$, so by Proposition \ref{superheavy_skeleton}, the skeleton is superheavy.
\end{proof}

\subsection{A Subcritical Example}
\begin{lem}
Take $M=\mathbb{CP}^2$, with the Fubini-Study symplectic form.
Let $[z_0:z_1:z_2]$ be homogeneous coordinates on $M$.

Let
\begin{equation}
s=z_1^n-z_2^n\in \mathcal{O}(n)
\end{equation}

Let $D=s^{-1}(0)$ and $X=M\setminus D$.

The skeleton of $(X,d^c\log ||s||)$ is the union of $n$ arcs inside $\{z_0=0\}$ described in Theorem \ref{arcs_nbhd}.
\end{lem}
\begin{proof}
Let $(p_i,\theta_i)$ be standard action-angle coordinates for $i=1,2$.

The primitive $d^c\log ||s||$ is invariant under the Hamiltonian $\mathbb{R}/\mathbb{Z}$-action
\begin{equation}
t\cdot [z_0:z_1:z_2]=[e^{-2\pi it}z_0:z_1:z_2]
\end{equation}
 which is generated by $p_1+p_2$, so by Lemma \ref{skeleton_barycentre}, the skeleton is contained in a fibre. The global minima of $-\log ||s||$ are contained in the fibre $\{p_1+p_2=1\}=\{z_0=0\}\cong\mathbb{CP}^1$, so the skeleton is also contained in this fibre.

With respect to the Fubini-Study metric,
\begin{equation}
||s||^2=p_1^n+p_2^n-2(p_1p_2)^{\frac{n}{2}}\cos(n(\theta_1-\theta_2)),
\end{equation}
so the critical points of $-\log ||s||$ are the global minima $[0:1:0]$ and $[0:0:1]$, and the saddle points $\{p_1=p_2=\frac{1}{2}, \theta_1-\theta_2=\frac{2\pi}{n}(k+\frac{1}{2})\}$ for $k=1,...,n$.

In the complement of the critical points, 
\begin{equation}
\nabla ||s||^2 \cdot (\theta_1-\theta_2)=-\pi n(p_1p_2)^{\frac{n-2}{2}}\sin(n(\theta_1-\theta_2))
\label{cos_n_-n}
\end{equation}
so $-\nabla \log ||s||\cdot \cos(n(\theta_1-\theta_2))\geq 0$. This shows that the arcs of the form $\{z_0=0, \theta_1-\theta_2=\frac{2\pi}{n}(k+\frac{1}{2})\}$ for $k=1,...,n$, are preserved by the gradient flow, and any gradient flowline converging to a saddle point must be contained in one of these arcs, so the skeleton is the union of these arcs as required.
\end{proof}

\begin{cor}[Proof of Theorem \ref{arcs_nbhd}]
For $n\geq 3$, the union of $n$ arcs in Equation \eqref{n_circles} is a retract of a neighbourhood $U\subset \mathbb{CP}^2$ which is superheavy, and occupies $\frac{1}{9}$ of the total volume of $\mathbb{CP}^2$.
\end{cor}
\begin{proof}
We have that
\begin{equation} 
-2K_{\mathbb{CP}^2}=\frac{6}{n}D.
\end{equation}

The divisor $D$ is a union of $n$ lines, intersecting at the point $[1:0:0]$. At this point, $s$ is quasihomogeneous with weight system $(1,1,n)$, so by Lemma \ref{qhsings}, there exists a system of commuting Hamiltonians for $V_D$ with
\begin{equation}
\sigma_{crit}=\max\left(\frac{3-n}{3},\frac{1}{3}\right).
\end{equation}

In particular, for $n\geq 3$, $\sigma_{crit}=\frac{1}{3}$. In this case, by Theorem \ref{superheavy_subset}, $\{\rho^0\leq \frac{1}{3}\}$, which has volume $\frac{1}{3^2}$ by construction, is superheavy.
\end{proof}

\subsection{Chiang Lagrangian}
\begin{lem}
Take $M=\mathbb{CP}^3$, with the Fubini-Study symplectic form.

Identifying $M$ with the projectivisation of the space of cubic polynomials, let $s\in\mathcal{O}(4)$ denote the discriminant.

Let $D=s^{-1}(0)$ and $M=X\setminus D$.

The skeleton of $(X,d^c\log ||s||)$ is the Chiang Lagrangian, as described in \cite{Chiang}.
\end{lem}
\begin{proof}
The algebraic group $SL(2,\mathbb{C})$ acts on $\mathbb{CP}^1$ by Möbius transformations. 

We can identify $\mathbb{CP}^3$ with $Sym^3\mathbb{CP}^1$, by viewing $\mathbb{CP}^3$ as the projectivisation of the space of degree $3$ polynomials. We therefore obtain an action of $SL(2,\mathbb{C})$ on $\mathbb{CP}^3$.

Let $\Delta=\{e^{\frac{2\pi i k}{3}}| k=0,1,2\}\subset \mathbb{CP}^1$ denote the 3rd roots of unity. Under stereographic projection, these are the vertices of an equilateral triangle on the equator of $S^2$. We can view $\Delta$ as a point in $\mathbb{CP}^3$. Because $SL(2,\mathbb{C})$ acts 3-transitively, $X$ is the affine $SL(2,\mathbb{C})$-orbit of $\Delta$.

The discriminant locus $D=s^{-1}(0)$ is the algebraic surface consisting of configurations where at least $2$ points coincide.

Let $N$ denote the singular locus of $D$. This is the twisted cubic in $\mathbb{CP}^3$, the smooth genus 0 curve consisting of configurations where all $3$ points coincide.

The action of $SU(2)\subset SL(2,\mathbb{C})$ on $M$ is Hamiltonian. Let $\mu:M\rightarrow \mathfrak{su}_2^*$ denote the moment map. Let $\pi^{-1}:\mathbb{CP}^1\rightarrow S^2$ denote the inverse of stereographic projection, where we view $S^2$ as the unit sphere in $\mathfrak{su}_2^*$. Then we have that
\begin{equation}
\mu\left(\{\alpha_1,\alpha_2,\alpha_3\}\right)=\frac{1}{3}\sum_{i=1}^3\pi^{-1}(\alpha_i).
\end{equation}

The Chiang Lagrangian, denoted $L\subset X$, is the $SU(2)$-orbit of $\Delta$, which is a smooth quotient of $SU(2)$, and equal to $\mu^{-1}(0)$. By Lemma \ref{skeleton_barycentre}, because $0$ is the only fixed point of the adjoint representation of $SU(2)$, $L$ is the Lagrangian skeleton of $(X,d^c\log ||s||)$.
\end{proof}

\begin{cor}[Proof of Theorem \ref{Chiang_nbhd}]
The Chiang Lagrangian is a retract of a neighbourhood $U\subset\mathbb{CP}^3$ which is superheavy, and occupies $\frac{1}{216}$ of the volume of $\mathbb{CP}^3$.
\end{cor}

\begin{proof}
By Lemma 1.2 of \cite{Seidel}, $||\mu||$ generates a Hamiltonian circle action on $M\setminus L$. This action preserves orbits of the $SU(2)$-action, and therefore preserves $D$. Furthermore, $||\mu||$ attains a global maximum, along $N=\{||\mu||=1\}$, so the action fixes $N$ pointwise. 

We may therefore define 
\begin{equation}
r_N=1-||\mu||,
\end{equation}
a radial Hamiltonian for $N\in V_D$, to which $d^c\log ||s||$ is adapted by construction.

We will now determine the isotropy representation of the action generated by $r_{N}$, at the point $x=\{0,...,0\}$.

The hyperplane $H=\left\{\{\alpha_1,\alpha_2,\alpha_3\}|\sum_{i=1}^3\alpha_i=0\right\}$ is orthogonal to $N$ at $x$. Let $\xi=\mu(x)$.  $H$ is the preimage of the ray in $\mathfrak{su}_2^*$ generated by $\xi$, so $||\mu|||_H=\xi^T\mu|_H$. The restriction of the circle action to $H$ is therefore the stabiliser of $x$ in $SU(2)$, which is the subgroup of the form $diag(e^{\frac{1}{2}it},e^{-\frac{1}{2}it})$, which acts via $\{\alpha_1,\alpha_2,\alpha_3\}\to \{e^{it}\alpha_1,e^{it}\alpha_2,e^{it}\alpha_3\}$. The tangent space
$T_xH=(T_xN_C)^\perp$ has an orthogonal basis given by the monomials $T^2,T^3$. With respect to this basis, the isotropy representation is 
\begin{equation}
e^{it}\mapsto diag(e^{2it},e^{3it}).
\end{equation}
The weight of $r_{N}$ is therefore given by $w_N=2+3=5$.

We can compute the rescaled action $\lambda_{N}$ by considering a disc $u$ of the form 
\begin{equation}
u(z)= \{z\alpha_1,z\alpha_2,z\alpha_3\}
\end{equation}
for some $\{\alpha_1,\alpha_2,\alpha_3\}\in W\cap H$.
\begin{equation}
u^*s=z^6\prod_{i\neq j}(\alpha_i-\alpha_j)
\end{equation}
so $\lambda_{N}=2[u]\cdot [Y]=12$.

By Lemma \ref{hypersurf_adapted}, $V_D$ admits a system of commuting Hamiltonians to which $d^c\log ||s||$ is weakly adapted, such that
\begin{equation}
\sigma_{crit}=\frac{12-10}{12}=\frac{1}{6}.
\end{equation}
By Theorem \ref{superheavy_subset}, $\{\rho^0\leq \frac{1}{6}\}$, which has volume $\frac{1}{6^3}$ by construction, is superheavy.
\end{proof}
\section{Appendix}
\subsection{Construction of Adapted Structures}
\begin{lem}
\label{average_construction}
Let $V$ be a stratified symplectic subvariety in $(M,\omega)$, and $r_v:U_v\rightarrow\mathbb{R}_{\geq 0}$ for $v\in V$ a system of commuting Hamiltonians for $V$.

Let $Y\subset X$ be an open set, $r_v$-invariant for all $v\in V$.

Let $\alpha\in\Omega^1(Y)$ be a one-form such that $d\alpha|_{Y\cap U_v}$ is $r_v$-invariant for $v\in V$.

Then there exists $\overline{\alpha}\in\Omega^1(Y)$, an $\omega$-compatible almost-complex structure $J$ on $Y$, and $r_v$-invariant stratum neighbourhoods $T_v\subset U_v$ for $v\in V$ such that $\overline{\alpha}|_{Y\cap T_v}$ and $J|_{Y\cap T_v}$ are $r_v|_{T_v}$-invariant, and $\overline{\alpha}=\alpha+df$. 

Furthermore, let $D\subset Y$ be a submanifold, $r_v$-invariant for $v\in V$. If $\alpha_x=0$ for $x\in D$, we may ensure that $\overline{\alpha}_x=0$ for $x\in D$. If $D$ is an almost-complex submanifold with respect to some compatible almost-complex structure on $Y$, then we may ensure that $D$ is an almost-complex submanifold with respect to $J$.

\end{lem}
\begin{proof}
We say that $\alpha$ and $J$ are $k$-adapted if there exist stratum neighbourhoods $T_v\subset U_v$ for $v\in V$ such that, for every $I\subset V$ such that $\# I=k$, $T_I$ is $r_I$-invariant and $\alpha|_{T_I\cap Y}$ and $J|_{T_I\cap Y}$ are $r_I|_{T_I}$-invariant.

If $k>\# V$, any $\alpha$ and $J$ are vacuously $k$-adapted.

Suppose $\alpha$ and $J$ are $k+1$-adapted, and let the stratum neighbourhoods $T_v\subset U_v$ be as above. By Lemma \ref{invariant_neighbourhoods}, we may assume $T_I$ is $r_I$-invariant for any $I\subset V$. We may also assume that $T_v\subset M$ is open for $v\in V$. 

Now for each $I\subset V$ such that $\# I=k$, let 
\begin{equation}
\alpha_I=\int_{t\in [0,1]^I}\left(\phi_{r_I}^t\right)^*\alpha|_{Y\cap T_I}dt.
\end{equation}
We can define $J_I$ similarly, using a standard polar decomposition argument. First let
\begin{equation}
A_I=\int_{t\in [0,1]^I}\left(\phi_{r_I}^t\right)^*J|_{Y\cap T_I}dt.
\end{equation}
We can define a metric $g_I(u,v)=\omega(u,A_Iv)$. Now with respect to $g_I$, $A_IA_I^T$ is positive definite and symmetric, and has a positive definite square root. Let $J_I=\left(\sqrt{A_IA^T_I}\right)^{-1}A_I$, which is $\omega$-compatible and $r_I$-invariant by construction.  

For $v\in V$, on restriction to $U_v$, by integrating Cartan's formula for the Lie derivative of the left hand side, we have
\begin{equation}
\label{Ham_pullback_exact}
(\phi_{r_v}^t)^*\alpha-\alpha=d\int_{\tau=0}^t\iota_{X_{r_v}}\left((\phi_{r_v}^\tau)^*\alpha-\alpha\right)d\tau,
\end{equation}
so $\alpha_I=\alpha|_{Y\cap T_I}+df_I$, where $f_I$ vanishes on $T_{I\cup \{v\}}$ for $v\in V\setminus I$, because $\alpha$ is already invariant in this region. Similarly, $A_I$ and therefore $J_I$ agrees with $J$ on this region. 

Furthermore, if $\alpha_x=0$ for $x\in D$, then $(\alpha_I)_x=0$ for $x\in D\cap T_I$. If $D$ is an almost-complex submanifold for $J$, then $D\cap T_I$ is an almost-complex submanifold for $J_I$, because $(A_I)_x$ respects the decomposition $T_xD\oplus T_xD^\omega$ for $x\in D\cap T_I$.

By Lemma \ref{invariant_neighbourhoods}, we can choose stratum neighbourhoods $S_v\subset T_v$ for $v\in V$ such that $S_v\cap X\subset X$ is closed for $v\in V$, and $S_I$ is $r_I$-invariant for any $I\subset V$. We can choose a smooth function $f$ on $Y$ such that $f|_{S_I\cap Y}=f_I|_{S_I\cap Y}$, for $\# I=k$. Let $\overline{\alpha}=\alpha+df$. Then $\overline{\alpha}|_{Y\cap S_I}=\alpha_I|_{Y\cap S_I}$. Similarly we can choose a compatible almost-complex structure $\overline{J}$ on $Y$ such that $\overline{J}|_{Y\cap S_I}=J_I$.

By construction, $\overline{\alpha}$  and $\overline{J}$ are $k$-adapted. By induction, there exist such $\alpha$ and $J$ which are $1$-adapted.
\end{proof}

\subsection{Equivariant Skeleta}
\begin{lem}
\label{skeleton_barycentre}
Suppose $(X,\theta)$ is a finite type convex symplectic manifold.

Let $G$ be a compact Lie group and $\mathfrak{g}$ its Lie algebra, and suppose
\begin{equation}
\mu:X\rightarrow \mathfrak{g}^*
\end{equation}
is a proper map which generates a hamiltonian $G$-action on $X$, which preserves $\theta$, and $\mu$ is equivariant with respect to the coadjoint representation.

Then the Lagrangian skeleton $L$ of $(X,\theta)$ is contained in a fibre $\mu^{-1}(\xi)$ for some fixed point $\xi$ of the adjoint representation, and $\psi^{-t}(\mu^{-1}(\xi))\subset \mu^{-1}(\xi)$ for all $t>0$.
\end{lem}

\begin{proof}
Let $\omega=d\theta$, $Z$ denote the Liouville vector field $\omega$-dual to $\theta$, and $\psi^{-t}$ the negative time-$t$ flow of $Z$ for $t>0$.

We have
\begin{equation}
d(\iota_Zd\mu-\mu)=d(\iota_{X_\mu}\theta-\mu)=\mathcal{L}_{X_\mu}\theta-\iota_{X_\mu}d\theta-d\mu=\mathcal{L}_{X_\mu}\theta=0,
\end{equation}
so $\iota_Zd\mu=\mu-\xi$ for some $\xi\in \mathfrak{g}^*$, and for $t>0$,
\begin{equation}
\mu\circ \psi^{-t}=\xi+e^{-t}(\mu-\xi).
\end{equation}
Therefore as $t\to \infty$, $\mu(\psi^{-t}(x))\to \xi$ for all $x\in X$, so
\begin{equation}
L=\bigcap_{t>0}\psi^{-t}(X)\subset \mu^{-1}(\xi).
\end{equation} 
Because $\psi^{-t}$ and $\mu$ are $G$-equivariant, $\xi$ must be a fixed point.

\end{proof}
\bibliography{refs}

\end{document}